\documentclass[12pt]{amsart}

\usepackage{amsmath}
\usepackage{amssymb}
\usepackage{mathdots}
\usepackage[all]{xy}

\numberwithin{equation}{section}

\newtheorem{thm}{Theorem}[section]

\newtheorem{lem}[thm]{Lemma}
\newtheorem{prop}[thm]{Proposition}
\newtheorem{cor}[thm]{Corollary}
\theoremstyle{definition}
\newtheorem{rem}[thm]{Remark}

\newtheorem{exam}[thm]{Example}
\newtheorem{exam-nota}[thm]{Example-Notation}
\newtheorem{nota}[thm]{Notation}
\newtheorem{dfn}[thm]{Definition}

\newtheorem{dfn-nota}[thm]{Definition-Notation}

\newtheorem{dfn-lem}[thm]{Lemma-Definition}

\newcommand{\beqa}{\begin{eqnarray*}}
\newcommand{\eeqa}{\end{eqnarray*}}

\newcommand{\id}{\mbox{${\rm id}$}}

\renewcommand{\Im}{\mbox{${\rm Im}$}}

\newcommand{\fa}{\mbox{${\mathfrak a}$}}
\newcommand{\ft}{\mbox{${\mathfrak t}$}}
\newcommand{\fk}{\mbox{${\mathfrak k}$}}
\newcommand{\fg}{\mbox{${\mathfrak g}$}}
\newcommand{\fq}{\mbox{${\mathfrak q}$}}
\newcommand{\fl}{\mbox{${\mathfrak l}$}}
\newcommand{\fs}{\mbox{${\mathfrak s}$}}
\newcommand{\fsl}{\mbox{${\fs\fl}$}}

\newcommand{\fh}{\mbox{${\mathfrak h}$}}
\newcommand{\fn}{\mbox{${\mathfrak n}$}}

\newcommand{\fp}{\mbox{${\mathfrak p}$}}
\newcommand{\fr}{\mbox{${\mathfrak r}$}}
\newcommand{\fb}{\mbox{${\mathfrak b}$}}
\newcommand{\fz}{\mbox{${\mathfrak z}$}}
\newcommand{\fm}{\mbox{${\mathfrak m}$}}
\newcommand{\fu}{\mbox{${\mathfrak u}$}}

\newcommand{\eps}{\epsilon}

\newcommand{\C}{\mbox{${\mathbb C}$}}

\newcommand{\Ad}{{\rm Ad}}

\newcommand{\fgl}{\mathfrak{gl}}

\newcommand{\B}{\mathcal{B}}
\newcommand{\Par}{\mathcal{P}}

\newcommand{\X}{\mathfrak{X}}
\newcommand{\fso}{\mathfrak{so}}
\newcommand{\fsp}{\mathfrak{sp}}
\newcommand{\codim}{\mbox{codim}}

\title{Eigenvalue coincidences and multiplicity free spherical pairs}

\author[M. Colarusso]{Mark Colarusso}
\address{Department of Mathematical Sciences, University of Wisconsin-Milwaukee, Milwaukee, WI, 53201}
\email{colaruss@uwm.edu}

\author[S. Evens]{Sam Evens}
\address{Department of Mathematics, University of Notre Dame, Notre Dame, IN, 46556}
\email{sevens@nd.edu}

\begin{document}
\maketitle
\begin{abstract}
 In recent work, we related the structure of subvarieties of $n\times n$
complex matrices defined by eigenvalue coincidences to
$GL(n-1,\mathbb{C})$-orbits on the flag variety of
$\mathfrak{gl}(n,\mathbb{C})$. In the first part of this paper, we extend these
results to the complex orthogonal Lie algebra
$\mathfrak{g}=\mathfrak{so}(n,\mathbb{C})$. In the second part of the paper, we
use these results to study the geometry and invariant theory of the 
$K$-action on $\mathfrak{g}$, in the cases where $(\mathfrak{g}, K)$ is 
$(\mathfrak{gl}(n,\mathbb{C}), GL(n-1,\C))$ or $(\mathfrak{so}(n,\mathbb{C}), SO(n-1,\C))$.
 We study the geometric quotient $\mathfrak{g}\to
\mathfrak{g}//K$ and describe the closed $K$-orbits on $\mathfrak{g}$ and
the structure of the zero fibre.  We also prove that for $x\in \mathfrak{g}$,
the $K$-orbit $\Ad(K)\cdot x$ has maximal dimension if and only if
the algebraically independent generators of the invariant ring
$\mathbb{C}[\mathfrak{g}]^{K}$ are linearly independent at $x$, which
extends a theorem of Kostant.   We give applications of our results
to the Gelfand-Zeitlin system.

\end{abstract}

\section{Introduction} 

This paper studies two related questions.  Let $x\in\fgl(n,\C)$ be an $n\times n$ complex matrix, and 
let $x_{\fk}\in\fgl(n-1,\C)$ be the $(n-1)\times (n-1)$ submatrix in the upper left corner of $x$.  In \cite{CEeigen}, 
we studied the subvariety of $\fgl(n,\C)$ consisting of matrices $x$ such that $x$ and $x_{\fk}$ have a specified 
number of eigenvalues in common.  In the first part of the paper, we extend these results from $\fgl(n,\C)$
to $\fso(n,\C)$.   In the second part of the paper, we use the results
from \cite{CEeigen} and the first part of the paper to study the action of $K$ on $\fg$ by conjugation in the two cases $(K=GL(n-1,\C), \fg=\fgl(n,\C))$
and $(K=SO(n-1,\C), \fg=\fso(n,\C)).$  By a theorem of Knop \cite{Knhc},
the algebra $\C[\fg]^K = \C[\fg]^G \otimes
\C[\fk]^K$ is a polynomial algebra. It follows that
the quotient morphism $\fg \to \fg//K$ can be identified with
a morphism $\Phi_n$ from $\fg$ to affine space, which is a partial
version of a morphism considered by Kostant and Wallach \cite{KW1, KW2}.  
We study this morphism, and as a consequence,
we determine explicitly the closed $K$-orbits on $\fg$ and the structure of the zero fibre.  We also prove a variant
of Kostant's theorem using linear independence of differentials to
characterize regular elements \cite{Kostant63}.  This variant of Kostant's theorem allows us to give a simpler 
definition of the strongly regular elements of $\fg$, which were
introduced by Kostant and Wallach to construct the Gelfand-Zeitlin integrable system on $\fg$ (\cite{KW1}, \cite{Col2}).  
We use this simpler definition to establish new results about the Gelfand-Zeitlin system on $\fso(n,\C)$.

In more detail, let $\fg=\fso(n,\C)$ and let $\fk\subset \fg$ be the 
symmetric subalgebra
$\fk=\fso(n-1,\C)$ fixed by an involution $\theta$ of $\fg$. Let $r_n$ and
$r_{n-1}$ be the ranks of $\fg$ and $\fk$ respectively.
Recall that if $a$ is an eigenvalue of $x \in \fso(n,\C)$, then $-a$ is
also an eigenvalue of $x$, and if $n$ is odd, then 
the eigenvalue $0$ of $x$ occurs an odd number of times.  Let 
 \begin{equation}\label{eq:spectrum}
\sigma(x)=\{\pm b_{1},\dots, \pm b_{r_{n}}\}
\end{equation}
be the eigenvalues of $x$, listed with multiplicity, except that if $n$
is odd, we only list the eigenvalue $0$ $2j$ times if it appears with
multiplicity $2j+1$.  We call $\sigma(x)$ the spectrum of $x$.
For $x\in\fg$, let $x=x_{\fk} + x_{\fp}$ where $x\in \fk$ and $x_{\fp} 
\in \fg^{-\theta}$, and let $\sigma(x_{\fk})=\{\pm a_{1},\dots, \pm a_{r_{n-1}}\}$ be the 
spectrum of $x_{\fk}$, regarded as an element of $\fso(n-1,\C)$. 
 We consider the \emph{eigenvalue coincidence} varieties
$\fg(\ge i)$ consisting
of $x\in \fg$ such that $\sigma(x)$ and $\sigma(x_{\fk})$ share at least
$2i$ elements, counting multiplicity.   More precisely, for $i=0, \dots,
r_{n-1}$,
\begin{equation}\label{eq:coincidences}
\fg(\geq i):=\{x\in\fg: b_{j_{m}}=\pm a_{k_{m}},\, m=1,\dots, i \mbox{ with } \, j_r\not= j_s\mbox{ and } k_r \not= k_s \mbox{ for } r\not= s \}.
\end{equation}
For $\fg=\fgl(n,\C)$ and $\fk=\fgl(n-1,\C)$ thought of as the $(n-1)\times (n-1)$ upper left corner of $\fg$, the analogous varieties were studied in \cite{CEeigen}.  We let $\sigma(x)$ denote the eigenvalues of $x$, and let 
$\sigma(x_{\fk})$ denote the eigenvalues of $x_{\fk}$ regarded as an
element of $\fgl(n-1,\C)$.  
In this case, the variety $\fg(\geq i)$ consists of elements $x$ such that 
$\sigma(x)$ and $\sigma(x_{\fk})$ share at least $i$ elements, counted
with multiplicity.


We make use of the following notation throughout.  
We denote the flag variety of a reductive Lie algebra $\fg$ by
$\B_{\fg}$, or by $\B$ when $\fg$ is understood.
For a Borel subalgebra $\fb$ of $\fg$, we denote its $K$-orbit in 
$\B$ by $Q=K\cdot \fb$.
We denote the $K$-saturation of $\fb$ in $\fg$ by $Y_{Q}:=\Ad(K)\fb=\{ \Ad(k)x:k\in K, x\in \fb \}$.  Note that the variety 
$Y_{Q}$ depends only on the $K$-orbit $Q$ in $\B$.  We prove the following
result.

\begin{thm}\label{thm:bigthm}
The irreducible component decomposition of the variety $\fg(\geq i)$ is given by 
\begin{equation}\label{eq:king}
\fg(\geq i)=\displaystyle\bigcup_{\mbox{codim}(Q)=i} \overline{Y_{Q}}.
\end{equation}
In particular, if $\fg=\fso(2l,\C)$ is type $D$ then the varieties 
$\fg(\geq i)$ are all irreducible. If $\fg=\fso(2l+1,\C)$ is type 
$B$ then $\fg(\geq i)$ is irreducible for $i=0,\dots, l-1$ and has exactly two irreducible components when $i=l$.    
\end{thm} 

Although this statement is similar to Theorem 1.1 of \cite{CEeigen}, the
proof requires some significant new ideas, because 
 computations analogous to those performed in \cite{CEeigen} for $\fgl(n,\C)$
are intractable for $\fso(n,\C)$.


In the second part of the paper, we consider the pairs $(G,K)$ given by $G=GL(n,\C)$, $K=GL(n-1,\C)$
and $G=SO(n,\C)$, $K=SO(n-1,\C)$, which are essentially the only
 symmetric pairs for which the branching rule for finite dimensional representations from $G$ to $K$ is multiplicity free.  Let $\fg$ and $\fk$ be the corresponding Lie algebras.    For each pair, let $\tilde{G}=G\times K$ and let
$K_{\Delta}$ be the diagonal embedding of $K$ in $\tilde{G}$.
It is standard that $K_{\Delta}$ is a spherical subgroup of $\tilde{G}$ (Proposition \ref{prop:isspherical}).  We consider the coisotropy representation of $K_{\Delta}$ on 
$\tilde{\fg}/\fk_{\Delta}$, which coincides with the adjoint
action of $K$ on $\fg$.   We say that $x\in \fg$ is 
\emph{$n$-strongly regular} if $x$ is in a $K$-orbit of maximal dimension.  We write the generators of
$\C[\fg]^K$ as $\{f_{n-1,1}, \dots, f_{n-1, r_{n-1}}; f_{n, 1}, \dots,
f_{n, r_n}\}$ (for $\fgl(i,\C)$, $r_i=i$).  
The following theorem extends a basic result of Kostant
\cite{Kostant63}.   
\begin{thm}\label{thm:kostext}[Theorem \ref{thm:Kostant}] An element $x\in \fg$ is $n$-strongly regular if and only if 
\begin{equation}\label{eq:introdiffs}
df_{n-1,1}(x)\wedge\dots\wedge df_{n-1, r_{n-1}}(x)\wedge df_{n,1}(x)\wedge\dots\wedge df_{n, r_{n}}(x)\neq 0. 
\end{equation}
\end{thm}

Using Theorem \ref{thm:kostext}, we show that the Zariski open set 
$\fg(0)=\{x\in\fg:\, \sigma(x)\cap\sigma(x_{\fk})=\emptyset\}$ consists entirely of $n$-strongly regular elements (Theorem \ref{thm:gzero}) and use
it to show that for $x\in\fg(0)$, the fibre $\Phi_{n}^{-1}(\Phi_{n}(x))$ is a single $K$-orbit (Corollary \ref{c:isgenericallyrad}).  Using this result along with Theorem \ref{thm:bigthm} and Theorem 3.7 of \cite{CEeigen}, we give explicit representatives for all closed $K$-orbits on $\fg$ (Theorem \ref{thm:closedKinfg}) and determine the nilfibre $\Phi^{-1}_{n}(0)$  (Theorem \ref{thm:nilfibre}).  In contrast to the case of $\fgl(n,\C)$, we show
that for $\fso(n,\C)$ the fibre $\Phi_n^{-1}(0)$ contains no $n$-strongly regular elements (Proposition \ref{prop:overlaps} and Corollary \ref{c:nonsreg}).   

  This work is motivated by our interest in the Gelfand-Zeitlin system,
which is a maximal Poisson commutative family $J_{GZ}$ in $\C[\fg]$ defined using
a family of subalgebras
 $\fg_1 \subset \fg_2 \subset \dots \subset \fg_{n-1} \subset \fg_n$,
where $\fg_i = \fgl(i,\C)$ in the general linear case, and $\fg_i=\fso(i,\C)$
in the orthogonal case.  Elements of $\fg$ for which the differentials $\{x\in\fg : df(x), \, f\in J_{GZ}\}$ are linearly independent are called \emph{strongly regular}.  In \cite{KW1}, Kostant and Wallach show that
any regular adjoint orbit in $\fgl(n,\C)$ contains strongly regular elements which implies that the Gelfand-Zeitlin system
is completely integrable on any regular adjoint orbit.  Using different techniques, the first author produced strongly regular elements in certain 
regular semisimple orbits of $\fso(n,\C)$ and proved the integrability of the Gelfand-Zeitlin system on these orbits \cite{Col2}.  The case of $\fgl(n,\C)$ is much better understood than the case of $\fso(n,\C)$, because the $\fso(n,\C)$
Gelfand-Zeitlin system is less amenable to computation.
 We hope that our methods will make the Gelfand-Zeitlin system
for $\fso(n,\C)$ more tractable to understand and will improve
our understanding of
the Gelfand-Zeitlin system for $\fg=\fgl(n,\C)$.
  As a step in this direction, we observe that 
Theorem \ref{thm:kostext} can be used to simplify the criterion for an element of $\fgl(n,\C)$ and $\fso(n,\C)$ to be strongly regular from \cite{KW1, Col2} (Proposition \ref{prop:fullsreg}).  This allows us to identify a previously unknown set of strongly regular elements of $\fso(n,\C)$ (Proposition \ref{prop:orthofgtheta}), and in later work, we will show that every regular adjoint of $\fso(n,\C)$
orbit contains strongly regular elements, implying the integrability of the Gelfand-Zeitlin system on all regular adjoint orbits of $\fso(n,\C)$.  We can also show that in contrast to the case of $\fgl(n,\C)$ there are no strongly regular elements $x\in\fso(n,\C)$ with the property that $f(x)=0$ for all $f\in J_{GZ}$ (Remark \ref{r:nosreg}).  These observations were previously inaccessible using the more computational methods of \cite{Col2, Col1}.  

We also plan
to apply results of this paper to study the Gelfand-Zeitlin modules
introduced by Drozd, Futorny, and Ovsienko \cite{DFO}, 
which are quantum analogues of the Gelfand-Zeitlin integrable systems.
  Our results in this paper develop parts of
a Kostant-Rallis theory \cite{KR} for the spherical pair $(\tilde{\fg}, K_{\Delta})$, and 
we expect it to play an important role in understanding a category of Harish-Chandra modules for these spherical pairs, especially through the study of
associated varieties.    Using an equivalence of categories analogous to the equivalence between category $\mathcal{O}$ and certain Harish-Chandra modules (\cite{BoBr}, Section 3.4), we plan to use $(\tilde{\fg}, K_{\Delta})$-modules in our future work 
to produce examples of Gelfand-Zeitlin modules and other closely related modules.  



This paper is organized as follows.  The first part of the paper comprises Sections 
\ref{ss:realization} and  \ref{ss:irredcomp}.  In Section \ref{ss:realization},
we establish a number of preliminary results.  In 
Section \ref{ss:irredcomp}, we prove Theorem \ref{thm:bigthm}.  The second part 
of the paper consists of Section \ref{ss:git}, in which we prove Theorem \ref{thm:kostext}, determine
the closed $K$-orbits on $\fg$, and discuss applications to strongly
regular elements.  In the appendix, we give an alternative, simpler proof of
a theorem of Knop in a special case.

We would like to thank Bertram Kostant, Nolan Wallach, and Jeb Willenbring for useful discussions relevant
to the subject of this paper.

\section{Preliminaries}\label{ss:realization}
We recall basic facts concerning orthogonal Lie algebras, and develop
some basic framework for the study of eigenvalue coincidence varieties.  
We also classify the $K=SO(n-1,\C)$-orbits on the flag variety $\B$ of $\fso(n,\C)$ and give 
explicit representatives for each orbit.


\subsection{Realization of Orthogonal Lie algebras}\label{ss:orthoreal}  
  We give explicit descriptions of standard Cartan subalgebras
and corresponding root systems of $\fso(n,\C)$.  Our exposition follows Chapters 1 and 2 of \cite{GW}.  

Let $\beta$ be the non-degenerate, symmetric 
bilinear form on $\C^{n}$ given by 
\begin{equation}\label{eq:beta}
\beta(x,y)=x^{T} S_{n} y, 
\end{equation}
where $x, y$ are $n\times 1$ column vectors and $S_{n}$ is the $n\times n$ 
matrix:
\begin{equation}\label{eq:Sn}
S_{n}=\left[\begin{array}{ccccc}
0& \dots &\dots & 0 & 1\\
\vdots &  & & 1 & 0\\
\vdots &  &\iddots  & & \vdots\\
0& 1& \dots & 0 & \vdots\\
1 & 0& \dots & \dots & 0\end{array}\right]
\end{equation}
with ones down the skew diagonal and zeroes elsewhere.  
The special orthogonal group is
$$
SO(n,\C):=\{g\in SL(n,\C): \, \beta(gx, gy)=\beta(x,y)\; \forall x,\, y \in \C^{n}\}.
$$
Its Lie algebra is
$$
\fso(n,\C)=\{Z\in\fgl(n,\C):\, \beta(Zx, y)=-\beta(x,Zy)\,\forall\; x, \, y\in\C^{n}\}.  
$$
For our purposes, it will be convenient to have explicit matrix descriptions of $\fso(n,\C)$.  
We consider the cases where $n$ is odd and even separately.  Throughout, we denote the standard basis of $\C^{n}$ by $\{e_{1},\dots, e_{n}\}$.


\subsubsection{Realization of $\fso(2l,\C)$} \label{ss:soevenreal}
Let $\fg=\fso(2l,\C)$ be of type $D$.  The subalgebra of diagonal matrices $\fh:=\mbox{diag}[a_{1},\dots, a_{l}, -a_{l},\dots, -a_{1}],\, a_{i}\in\C$ is a Cartan subalgebra of $\fg$.  We refer to $\fh$ as the \emph{standard Cartan subalgebra}.  
Let $\epsilon_{i}\in\fh^{*}$ be the linear functional $\epsilon_{i}(\mbox{diag}[a_{1},\dots, a_{l}, -a_{l},\dots, -a_{1}])=a_{i}$, and 
let $\Phi(\fg, \fh)$ be the roots of $\fg$ with respect to $\fh$.  It is well-known that 
\begin{equation}\label{eq:soevenroots}
\Phi(\fg,\fh)=\{\epsilon_{i}-\epsilon_{j},\, \pm(\epsilon_{i}+\epsilon_{j}):\; 1\leq i\neq j\leq l\}. 
\end{equation}
We take as our \emph{standard positive roots} the set:
\begin{equation}\label{eq:soevenposroots}
\Phi^{+}(\fg,\fh):=\{\epsilon_{i}-\epsilon_{j},\, \epsilon_{i}+\epsilon_{j}:\; 1\leq i< j\leq l\}. 
\end{equation}
with corresponding simple roots 
\begin{equation}\label{eq:sosimpleroots}
\Pi:=\{\alpha_{1},\dots, \alpha_{l-1}, \alpha_{l}\}\mbox{ where } \alpha_{i}=\epsilon_{i}-\epsilon_{i+1}, \, i=1, \dots, l-1, \, \alpha_{l}=\epsilon_{l-1}+\epsilon_{l}.
\end{equation}
The \emph{standard} Borel subalgebra $\fb_{+}:=\displaystyle\bigoplus_{\alpha\in\Phi^{+}(\fg,\fh)} \fg_{\alpha}$
is easily seen to be the set of upper triangular matrices in $\fg$. 

For the purposes of computations with $\fso(2l,\C)$, it is convenient to 
relabel part of the standard basis of $\C^{n}$ as $e_{-j}:=e_{2l+1-j}$ for $j=1, \dots, l$.  

\subsubsection{Realization of $\fso(2l+1,\C)$} \label{ss:sooddreal}
Let $\fg=\fso(2l+1,\C)$ be of type $B$.  The subalgebra of diagonal matrices $\fh:=\mbox{diag}[a_{1},\dots, a_{l}, 0, -a_{l},\dots, -a_{1}],\, a_{i}\in\C$ is a Cartan subalgebra of $\fg$.  We again refer to $\fh$ as the \emph{standard Cartan subalgebra}.  
Let $\epsilon_{i}\in\fh^{*}$ be the linear functional $\epsilon_{i}(\mbox{diag}[a_{1},\dots, a_{l},0,  -a_{l},\dots, -a_{1}])=a_{i}$.  In this case, we have 
\begin{equation}\label{eq:sooddroots}
\Phi(\fg,\fh)=\{\epsilon_{i}-\epsilon_{j},\, \pm(\epsilon_{i}+\epsilon_{j}):\; 1\leq i\neq j\leq l\} \cup \{\pm \epsilon_{k} :\, 1\leq k\leq l\}.
\end{equation}
We take as our \emph{standard positive roots} the set:
\begin{equation}\label{eq:sooddposroots}
\Phi^{+}(\fg,\fh):=\{\epsilon_{i}-\epsilon_{j},\, \epsilon_{i}+\epsilon_{j}:\; 1\leq i< j\leq l\}\cup \{ \epsilon_{k} :\, 1\leq k\leq l\} . 
\end{equation}
with corresponding simple roots 
\begin{equation}\label{eq:sooddsimpleroots}
\Pi:=\{\alpha_{1},\dots, \alpha_{l-1}, \alpha_{l}\}\mbox{ where } \alpha_{i}=\epsilon_{i}-\epsilon_{i+1}, \, i=1, \dots, l-1, \, \alpha_{l}=\epsilon_{l}.
\end{equation}
The \emph{standard} Borel subalgebra $\fb_{+}:=\displaystyle\bigoplus_{\alpha\in\Phi^{+}(\fg,\fh)} \fg_{\alpha}$
is easily seen to be the set of upper triangular matrices in $\fg$.
 
   As for $\fso(2l+1,\C)$, we relabel the standard basis of $\C^{n}$ by
letting $e_{-j}:=e_{2l+2-j}$ for $j=1, \dots, l$ and $e_{0}:=e_{l+1}$.  

\subsection{Real Rank $1$ symmetric subalgebras}\label{ss:symmetricreal}

For later use, 
recall the realization of $\fso(n-1,\C)$ as a symmetric 
subalgebra of $\fso(n,\C)$.
For $\fg=\fso(2l+1, \C)$, let $t$ be an element of the Cartan subgroup with
Lie algebra $\fh$  with the property that 
$\Ad(t)|_{\fg_{\alpha_i}}=\id$ for $i=1, \dots, l-1$ and $\Ad(t)|_{\fg_{\alpha_l}}=-\id$.  Consider the involution $\theta_{2l+1}:=\Ad(t)$.   Then $\fk=\fso(2l,\C)=\fg^{\theta_{2l+1}}$ 
(see \cite{Knapp02}, p. 700).  Note that $\fh \subset \fk$.
In the case $\fg=\fso(2l,\C)$, $\fk=\fso(2l-1,\C)=\fg^{\theta_{2l}}$,
where $\theta_{2l}$ is the involution induced by the diagram automorphism interchanging
the simple roots $\alpha_{l-1}$ and $\alpha_l$ (see \cite{Knapp02}, p. 703).   Note that
in this case, $\theta_{2l}(\eps_l)=-\eps_l$ and $\theta_{2l}(\eps_i)=\eps_i$ for $i=1, \dots, l-1$.  
We will omit the subscripts $2l+1$ and $2l$ from $\theta$
when $\fg$ is understood.

We also denote the corresponding involution of $G=SO(n,\C)$ by $\theta$.  
The group $G^{\theta}=S(O(n-1,\C)\times O(1,\C))$ is disconnected.  
We let $K:=(G^{\theta})^{0}$ be the identity component of $G^{\theta}$.  Then $K=SO(n-1,\C)$, and $\mbox{Lie}(K)=\fk=\fg^{\theta}$. 

\subsection{Notation}\label{s:thenotation}
We now lay out some of the notation that we will use throughout Sections \ref{ss:realization} and \ref{ss:irredcomp}. 
\begin{nota}\label{nota:thenotation}
\begin{enumerate}
\item We let $\fg=\fso(n,\C)$ and $\fk=\fso(n-1,\C)$ be the symmetric subalgebra given in Section \ref{ss:symmetricreal}, unless otherwise mentioned.  
It will also be convenient at times to denote the Lie algebra $\fso(i,\C)$ by $\fg_{i}$.
\item We let $r_{i}$ be the rank of $\fg_{i}$.  
\item For $x\in\fso(n,\C)$, we let $x_{\fk}$ the projection of $x$ onto $\fk$ off $\fg^{-\theta}$, the $-1$-
eigenspace of $\theta$.  
\item For any Lie algebra $\fg$, we denote by $\C[\fg]$, the ring of polynomial functions on $\fg$ and by 
$\C[\fg]^{G}$ the ring of adjoint invariant polynomial functions on $\fg$.
\end{enumerate}
\end{nota}

\subsection{The partial Kostant-Wallach map}



For $i=n-1, n$, let $\chi_{i}:\fg_{i}\to \C^{r_{i}}$ be the adjoint quotient. 
We define the \emph{partial Kostant-Wallach map} to be 
\begin{equation}\label{eq:partial}
\begin{array}{c}
\Phi_{n}:\fg\to \C^{r_{n-1}}\oplus \C^{r_{n}}, \\
\\
\; \Phi_n(x)=(\chi_{n-1}(x_{\fk}), \chi_{n}(x))=(f_{n-1, 1}(x_{\fk}),\dots, f_{n-1, r_{n-1}}(x_{\fk}), f_{n,1}(x),\dots, f_{n, r_{n}}(x)), 
\end{array}
\end{equation}
where $\C[\fg_{i}]^{G_{i}}=\C[f_{i,1},\dots, f_{i, r_{i}}]$.


\begin{prop}\label{prop:flat}
\begin{enumerate}

\item $\C[\fg]^K=\C[\fg]^G \otimes \C[\fk]^K$.
\item $\Phi_n$ coincides with the invariant theory quotient
morphism $\fg \to \fg//K.$   In particular, $\Phi_n$ is
surjective.
\item The morphism $\Phi_{n}$ is flat.  In particular, its fibres
are equidimensional varieties of dimension $\dim \fg- r_{n}-r_{n-1}$. 
\end{enumerate}
\end{prop}

\begin{proof} Recall the well-known fact that the fixed point algebra $U(\fg)^K$
of $K$ in the enveloping algebra $U(\fg)$ is commutative \cite{Johnson}.   Hence, $U(\fg)^K$ coincides with its centre,
$Z(U(\fg)^K)$.   In Theorem 10.1 of \cite{Knhc}, Knop shows that
$Z(U(\fg)^K) \cong
U(\fg)^G \otimes_{\C} U(\fg)^K$.  The first assertion now follows by
taking the associated graded algebra with respect to the usual filtration
of $U(\fg)$.  By the first assertion, $\Phi_n$ coincides with the
invariant theory quotient $\fg \to \fg//K$, which gives the second
assertion.   Note that if we embed
$\fk$ diagonally in $\fg \times \fk$, then $\fk^{\perp} \cong \fg$,
and this isomorphism is $K$-equivariant.  Then the flatness
of $\Phi_n$ follows by Korollar 7.2 of \cite{Kn}, which gives a criterion
for flatness of invariant theory quotients in the setting of
 spherical homogeneous spaces (see also \cite{Pancoiso}).
\end{proof}

\begin{rem}\label{r:glflat}
For $\fg=\fgl(n,\C)$ and $\fk=\fgl(n-1,\C)$ thought of as the subalgebra 
of $(n-1)\times (n-1)$ matrices in the left hand corner of $\fg$, Proposition
\ref{prop:flat} is also true by the same proof.  Here $K=GL(n-1,\C)$ is the algebraic subgroup of $GL(n,\C)$ 
corresponding to $\fk$. In this case, we proved Proposition \ref{prop:flat} (3) by more elementary means
in \cite{CEeigen}.
 In the appendix, we use conormal geometry
to give a more elementary proof of Korollar 7.2 of \cite{Kn} for a class of spherical
varieties, which applies to our setting.
\end{rem}

\begin{cor}\label{c:orthoGZ}
$\C[\fg]$ is a free $\C[\fg]^K$-module.
\end{cor}

\begin{proof} This follows by Lemma 2.5 of \cite{CEeigen} and the above
Proposition \ref{prop:flat} (3).
\end{proof}

We proved the analogous result for $\fgl(n,\C)$ in Proposition 2.6 of
\cite{CEeigen}, and noted that it follows from a result of
Futorny and Ovsienko, which states that $U(\fgl(n,\C))$ is free over the Gelfand-Zeitlin subalgebra \cite{Ov, Futfilt}.  It is not known whether the corresponding statement is true in the orthogonal case.   Our result that $\C[\fso(n,\C)]$ is free over $\C[\fso(n,\C)]^{SO(n-1,\C)}$ is a natural first step towards extending the result of Futorny and Ovsienko to the orthogonal setting and will be important in studying Gelfand-Zeitlin modules for the $\fso(n,\C)$.




\subsection{General Properties of eigenvalue coincidence varieties $\fg(\geq i)$}\label{ss:fgi}
In this section, we develop some fundamental facts about the eigenvalue coincidence varieties
discussed in the introduction (see (\ref{eq:coincidences})).  

 Let $\fh$ be the Cartan subalgebra of diagonal matrices in $\fg$, and let 
$\fh_{\fk}\subset\fk$ be the Cartan subalgebra of diagonal matrices in $\fk$.  We denote elements
of $\fh_{\fk}\times\fh$ by $(a, b)$, where $a=(a_{1},\dots, a_{r_{n-1}})$ and $b=(b_{1},\dots, b_{r_{n}})$ 
represent the diagonal coordinates of $a\in\fh_{\fk}$ and $b\in\fh$ as in Section \ref{ss:orthoreal} above.  Let $W=W(\fg,\fh)$ be the Weyl group of $\fg$, and let $W_K=W(\fk,\fh_{\fk})$ be the Weyl group of $\fk$.  
For $i=1,\dots, r_{n-1}$ define: 
$$
(\fh_{\fk}\times \fh)(\geq i):=\{(a,b): \exists \, v\in W_K, \, u\in W \mbox{ such that } (v\cdot a)_{j}=(u\cdot b)_{j},\, j=1,\dots, i\}.  
$$
We note that 
$(\fh_{\fk}\times \fh)(\geq i)$ is a $W_K\times W$-invariant closed subvariety of $\fh_{\fk}\times\fh$ 
and is equidimensional of codimension $i$.  
Let $p_G:\fh\to \fh/W$ and $p_K:\fh_{\fk} \to \fh_{\fk}/W_K$ 
be the invariant theory quotients.
Consider the finite morphism $p:=p_K\times p_G: (\fh_{\fk}\times\fh)\to  (\fh_{\fk}\times\fh)/ (W_K\times W)$.  
Let $F_{G}: \fh/W \to \C^{r_{n}}$ and $F_{K}: \fh_{\fk}/W_{K}\to \C^{r_{n-1}}$ be the Chevalley isomorphisms, and let 
$V^{r_{n-1},r_{n}}:=\C^{r_{n-1}}\times\C^{r_{n}}$, so that 
$F_{K}\times F_{G}: (\fh_{\fk}\times\fh)/ (W_K\times W) \to V^{r_{n-1},r_{n}}$ is an isomorphism.  The following varieties play a major role in our study of orthogonal eigenvalue coincidences.

\begin{dfn}\label{ref:thenotation}
For $i=0,\dots, r_{n-1}$, we let 
\begin{equation}\label{eq:Vgel}
V^{r_{n-1},r_{n}}(\geq i):=(F_{K}\times F_{G})((\fh_{\fk}\times\fh)(\geq i)/(W_K\times W)),
\end{equation} 
 \begin{equation}\label{eq:Vl}
  V^{r_{n-1},r_{n}}(i):=V^{r_{n-1},r_{n}}(\geq i)\setminus V^{r_{n-1},r_{n}}(\geq i+1).
  \end{equation}
For convenience, we let $V^{r_{n-1},r_{n}}( r_{n-1}+1)=\emptyset$.
\end{dfn}

\begin{lem}\label{lem:vnirreducible}
The set $V^{r_{n-1},r_{n}}(\geq i)$ is an irreducible closed subvariety of $V^{r_{n-1},r_{n}}$ of
dimension $r_{n}+r_{n-1}-i$.   Further, $V^{r_{n-1},r_{n}}(i)$ is open and dense 
in $V^{r_{n-1},r_{n}}(\geq i)$.
\end{lem}

\begin{proof}
Indeed, the set 
\begin{equation}\label{eq:Ydfn}
Y:= \{(a,b)\in \fh_{\fk}\times\fh : a_{j}=b_{j} \mbox{ for }
j=1, \dots, i\}
\end{equation}
is closed and irreducible of dimension $r_{n}+r_{n-1}-i$.   The
first assertion follows 
 since $(F_{K}\times F_{G})\circ p$ is a finite morphism and 
$(F_{K}\times F_{G})\circ p (Y)=V^{r_{n-1}, r_{n}}(\geq i)$.
The last assertion of the lemma now follows from Equation (\ref{eq:Vl}).
\end{proof}
We define 
\begin{equation}\label{eq:gidfn}
\fg(\geq i):=\Phi_{n}^{-1}(V^{r_{n-1}, r_{n}} (\geq i)), 
\end{equation}
(recall Equation (\ref{eq:partial}) for the definition $\Phi_n$).

For $i=0,\dots, r_{n-1}$, we define
 \begin{equation}\label{eq:fgl}
  \fg(i):=\fg(\geq i) \setminus \fg(\geq i+1) = \Phi_n^{-1}(V^{r_{n-1},r_{n}}(i)).
  \end{equation}
 Note that we have a partition of $\fg$ into disjoint locally closed sets:
 \begin{equation}\label{eq:fgipart}
 \fg=\displaystyle\bigcup_{i=0}^{r_{n-1}} \fg(i). 
 \end{equation}
  
  \begin{rem}\label{r:thesame}
We show that the definition of $\fg(\geq i)$ in (\ref{eq:gidfn}) agrees 
  with the one we gave in (\ref{eq:coincidences}).  We recall that 
  $\Phi_{n}=(\chi_{n-1},\chi_{n}),$ where $\chi_{i}:\fso(i,\C)\to \C^{r_{i}}$ is the 
  adjoint quotient.  Let $\fg=\fso(2l,\C)$ and let $x\in \fg$ satisfy the
property of Equation (\ref{eq:coincidences}).   Then there is $g\in G$
and $k\in K$ such that $\Ad(g)x$ and $\Ad(k)x_{\fk}$ are upper
triangular, $\Ad(g)x$ has diagonal part $\mbox{diag}[b_1, \dots, b_{l}, -b_{l}, \dots, -b_1]$, $\Ad(k)x$ has diagonal part
$\mbox{diag}[a_{1},\dots, a_{l-1}, 0, 0,-a_{l-1},\dots, -a_{1}]$, and
$b_{j_1}=\pm a_{k_1}, \dots, b_{j_i}=\pm a_{k_i}$.  We claim $\Phi_{n}(x)\in V^{r_{n-1}, r_{n}}(\geq i)$.  
Note that $\Phi_{n}(x)=(F_{K}\times F_{G})\circ p((a_{1},\dots, a_{l-1} ), (b_{1},\dots, b_{l})).$  Since $W$ contains
the subgroup $S_l$, and $W_K$ contains the subgroup $S_{l-1}$, we have
 $$p((a_{1},\dots, a_{l-1} ), (b_{1},\dots, b_{l}))=p((a_{1},\dots, a_{l-1}), (\pm a_{1},\dots, \pm a_{i}, b_{i+1},\dots, b_{l})).$$  Since $W_K$
contains all sign changes of the coordinates of $\fh_{\fk}$, it follows that that 
\begin{equation*}
\begin{split}
 p(( a_{1},\dots,  a_{i}, \dots,  a_{l-1}), (\pm a_{1},\dots, \pm a_{i}, b_{i+1},\dots, b_{l}))&= \\
p((\pm a_{1},\dots,\pm a_{i},\dots, a_{l-1}), ( \pm a_{1},\dots, \pm a_{i}, b_{i+1},\dots, b_{l})).
\end{split}
\end{equation*}
  It now follows that 
$\Phi_{n}(x)\in (F_{K}\times F_{G}) \circ p(Y)=V^{r_{n-1}, r_{n}}(\geq i)$,  where $Y$ is the variety defined in (\ref{eq:Ydfn}).  Thus, $x\in \fg(\ge i)$.  We leave the converse to the reader.  The case of $\fg=\fso(2l+1, \C)$ follows by similar reasoning.
  \end{rem}
  

  
  We now use the flatness of the Kostant-Wallach morphism asserted in Proposition \ref{prop:flat}
  to study the varieties $\fg(\geq i)$.

\begin{prop}\label{prop_dimgl}
\begin{enumerate}
\item The variety $\fg(\ge i)$ is equidimensional of dimension $\dim \fg-i$.
\item $ \overline{\fg(i)} = \fg(\ge i) = \displaystyle\bigcup_{k\geq i} \fg(k)$.
\end{enumerate}
\end{prop}

\begin{proof}
By Proposition \ref{prop:flat}, the morphism $\Phi_n$ is flat.
By Proposition III.9.5 and Corollary III.9.6 of \cite{Ha}, the variety
$\fg(\geq i)$ is equidimensional of dimension 
$\dim (V^{r_{n}-1,r_{n}}(\geq i))+\dim \fg-r_{n}-r_{n-1}$, which
gives the first assertion by Lemma \ref{lem:vnirreducible}.
 For the second assertion, 
by the flatness of $\Phi_n$, Theorem VIII.4.1 of \cite{SGA1},
and Lemma \ref{lem:vnirreducible},
\begin{equation}\label{eq:glbar}
\overline{\fg(i)}=\overline{\Phi_{n}^{-1}(V^{r_{n-1},r_{n}}(i))}=\Phi_{n}^{-1}(\overline{V^{r_{n-1},r_{n}}(i)})=\Phi_{n}^{-1}(V^{r_{n-1},r_{n}}(\geq i))=\fg(\geq i).                        
 \end{equation}
The remaining equality follows since $V^{r_{n-1},r_{n}}(\ge i)=\displaystyle \cup_{k\ge i} V^{r_{n-1},r_{n}}(k).$
\end{proof}

\subsection{The varieties $Y_{Q}$} \label{ss:YQ}
We now study the geometry of the varieties $Y_{Q}=\Ad(K)\fb$ for a $K$-orbit $Q=K\cdot \fb$ in $\B$.  
We begin by studying more general objects $Y_{Q_{\fr}}$, where $Q_{\fr}$ is a $K$-orbit in a partial flag variety. 
 
For a parabolic subgroup $P\subset G$ with Lie algebra $\fp\subset\fg$,
consider the partial Grothendieck resolution
$\tilde{\fg}^{\fp}=\{ (x,\fr)\in \fg \times G/P \; | \; x\in\fr\}$,
as well as the morphisms  
$\mu:\tilde{\fg}^{\fp}\to \fg,\; \mu(x,\fr)=x$, and $\pi: \tilde{\fg}^{\fp}\to G/P,\; \pi(x,\fr)=\fr$.  
 For $\fr \in G/P$, let $Q_{\fr}=K\cdot \fr \subset G/P$.
It is well-known that $\pi$ is a smooth morphism of relative dimension $\dim \fp$, and $\mu$ is proper 
with generically finite restriction to $\pi^{-1}(Q_{\fr})$ (see p. 622 of \cite{CEeigen}).
Thus, $\pi^{-1}(Q_{\fr})$ has dimension $\dim(Q_{\fr}) + \dim(\fr)$.

\begin{nota} For a parabolic subalgebra $\fr$ with $K$-orbit $Q_{\fr}\subset G/P$, we consider the irreducible
subset
\begin{equation}\label{eq:Yr}
Y_{\fr}:=\mu(\pi^{-1}(Q_{\fr}))=\Ad(K)\fr.
\end{equation}
$Y_{\fr}$ depends only on $Q_{\fr}$, and we will also denote this set as
\begin{equation}\label{eq:YQps}
Y_{Q_{\fr}}:=Y_{\fr}.
\end{equation}
\end{nota}

 It follows from generic finiteness of $\mu$ that $Y_{Q_{\fr}}$
contains an open subset of dimension
\begin{equation}\label{eq:YQdim}
\dim(Y_{Q_{\fr}}):=\dim \pi^{-1}(Q_{\fr})=\dim \fr+\dim(Q_{\fr})=\dim\fr+\dim (\fk/ \fk\cap \fr).
\end{equation}

\begin{rem} \label{r:closedorbitcase}
Since $\mu$ is proper, the set $Y_{Q_{\fr}}$ is closed when 
$Q_{\fr}=K\cdot \fr$ is closed in $G/P$.
\end{rem}
\begin{rem}\label{r:Qstrat}
Note that
$$
\fg=\bigcup_{Q\subset G/P} Y_{Q},
$$
where the union 
is taken over the finitely many $K$-orbits in $G/P$. 
\end{rem}

\begin{lem} \label{l:YQclosure}
Let $Q \subset G/P$ be a $K$-orbit.   Then
\begin{equation}
 \overline{Y_{Q}}=\bigcup_{Q^{\prime}\subset \overline{Q}} Y_{Q^{\prime}}.
\end{equation}
\end{lem}

\begin{proof}
Since $\pi$ is a smooth morphism, it is flat by Theorem III.10.2 of \cite{Ha}.  Thus,
by Theorem VIII.4.1 of \cite{SGA1}, 
$\pi^{-1}(\overline{Q})=\overline{\pi^{-1}(Q)}$.   The result follows
since $\mu$ is proper.
\end{proof}


\begin{prop}\label{prop:dimYQ}
Let $Q=K\cdot \fb$ be a $K$-orbit in $\B$ with $\mbox{codim}(Q)=i$.  Then 
\begin{equation}\label{eq:dimYQ}
\dim Y_{Q}=\dim \fg(\geq i).
\end{equation}
\end{prop}
\begin{proof}
By Equation (\ref{eq:YQdim}), it follows that 
$$\dim Y_{Q}=\dim Q+\dim \fb = \dim(\B) - i + \dim(\fb) = \dim(\fg) - i.$$ 
The assertion follows by part (1) of Proposition \ref{prop_dimgl}.
\end{proof}


Let $\codim(Q)=i$. To see that $\overline{Y_{Q}}$ is an 
irreducible component of $\fg(\geq i)$,  
it remains to show that $\overline{Y_{Q}}\subset\fg(\geq i)$.  For this, it is convenient to replace the $K$-orbit $Q$ in $\B$ 
with a $K$-orbit $Q_{\fr}$ of a $\theta$-stable parabolic subalgebra 
$\fr\supset \fb$ in a partial flag variety $G/P$.  We will show that $\fr\in G/P$ can be chosen 
so that $\overline{Y_{Q}}=Y_{Q_{\fr}}$, and $Y_{Q_{\fr}}\subset \fg(\geq i)$.  
The first step is to relate the geometry of $Q$ and $Q_{\fr}$ for a general $\theta$-stable $\fr$ with $\fr\supset\fb$ and develop a necessary condition 
for $\overline{Y_{Q}}=Y_{Q_{\fr}}$.  Let $R$ be the parabolic subgroup of $G$ with Lie algebra $\fr$.  Consider the canonical fibre bundle:
$$
R/B\to \B\stackrel{p}{\to} G/R,
$$
which induces a bundle 
\begin{equation}\label{eq:Qfibr}
Q\cap p^{-1}(Q_{\fr})\to Q\to Q_{\fr}
\end{equation}
over the $K$-orbit $Q_{\fr}$.
To study the fibre bundle (\ref{eq:Qfibr}), we consider the $\theta$-stable
parabolic subalgebra $\fr$ in more detail.  
It follows from Theorem 2 of \cite{BH} that $\fr$ has a $\theta$-stable Levi decomposition 
$\fr=\fl\oplus\fu$.  The Levi subalgebra decomposes as $\fl=\fz\oplus\fl_{ss}$, 
with the centre $\fz$ and the semisimple part $\fl_{ss}=[\fl,\fl]$ both $\theta$-stable.  Further, $\fk\cap\fr$ is a parabolic 
subalgebra of $\fk$ with Levi decomposition 
$$\fk\cap\fr=\fk\cap\fl\oplus\fk\cap\fu=\fk\cap\fz\oplus\fk\cap\fl_{ss}\oplus\fk\cap\fu.$$
  The corresponding parabolic subgroup $R$ is also $\theta$-stable, and 
$K\cap R$ is a parabolic subgroup of $K$ with Levi decomposition 
$(K\cap Z)\cdot (K\cap L_{ss})\cdot K\cap U$ (see Theorem 2, \cite{BH}).  
In particular, $Q_{\fr}\cong K/(K\cap R)$ is closed.  Recall that 
$R/B  \cong \B_{\fl_{ss}}$.  Thus, the fibre bundle 
 (\ref{eq:Qfibr}) gives the $K$-orbit $Q$ on $\B$ the structure of a $K$-homogeneous 
 fibre bundle over the closed $K$-orbit $Q_{\fr}=K\cdot\fr$ in $G/R$ with fibre 
 the $K\cap L_{ss}$-orbit of $\fb$ in $R/B\cong \B_{\fl_{ss}}$, i.e.:
\begin{equation}\label{eq:bundle}
Q\cong K\times_{K\cap R} (K\cap L_{ss} )\cdot \fb.
\end{equation}

\begin{prop}\label{prop:dimYQs}
Suppose that the orbit $(K\cap L_{ss})\cdot \fb$ in (\ref{eq:bundle}) is open in $R/B\cong \B_{\fl_{ss}}$.  
Then $\dim Y_{\fb}=\dim Y_{\fr}$.  Further, $Y_{\fr}$ is a closed, irreducible subvariety of $\fg$, 
so that $\overline{Y_{\fb}}= Y_{\fr}$.
\end{prop}  
\begin{proof}
Indeed,
\begin{equation}
\begin{split}
\dim Y_{\fr}&=\dim Q_{\fr}+\dim \fr\;  (\mbox{by } (\ref{eq:YQdim}))\\
&=\dim Q-\dim \B_{\fl_{ss}} +\dim \fr \; (\mbox{by } (\ref{eq:bundle})) \\
&=\dim Q-\dim \B_{\fl_{ss}}+\dim \B_{\fl_{ss}}+\dim \fb\\
&= \dim Q+\dim\fb\\
&= \dim Y_{\fb} \;(\mbox{by } (\ref{eq:YQdim})). 
\end{split}
\end{equation}
 It follows from definitions that $Y_{\fb}\subset Y_{\fr}$.  Since $Q_{\fr}$ is closed, $Y_{\fr}$ is closed by Remark \ref{r:closedorbitcase}.  Thus, $\overline{Y_{\fb}}=Y_{\fr}$ since $Y_{\fr}$ is irreducible. 
\end{proof}

In Theorems \ref{thm:YQs2} and \ref{thm:YQs1}, we show that
 for any Borel subalgebra $\fb \in \B$ whose orbit $K\cdot \fb$
has codimension $i$,
there is a $\theta$-stable parabolic subalgebra $\fr$ with $\fb\subset\fr$ such 
that the hypothesis of Proposition \ref{prop:dimYQs} is satisfied, and $Y_{Q_{\fr}}\subset \fg(\geq i)$.  To do this, 
we need to classify the $K$-orbits on $\B$ and develop explicit descriptions of representatives
of the $K$-orbits on $\B$. 

\begin{rem}\label{r:itholds}
When $\fg=\fgl(n,\C)$ and $\fk=\fgl(n-1,\C)\oplus\fgl(1,\C)$ is the symmetric subalgebra of block diagonal matrices, we have shown that
 for any Borel subalgebra $\fb\subset\fg$, there is a $\theta$-stable parabolic subalgebra $\fr$ with $\fb\subset\fr$ so that $(K\cap L_{ss})\cdot \fb$ is open in $\B_{\fl_{ss}}$.  This is implicit in Lemma 3.5 of \cite{CEKorbs} and in the computations of Proposition 2.15 of \cite{CEeigen}.
\end{rem}




\subsection{Description of $K$-orbits on $\B$ in the orthogonal case}\label{ss:Korbits}

We classify the $K$-orbits on $\B$ for $\fg=\fso(n,\C)$ and $\fk=\fso(n-1,\C)$.  In particular,
we explain how to recover the orbit diagrams from Figure 4.3 of 
\cite{Collingwood}, 
but also give explicit representatives of each orbit for later use.  

We begin by recalling some generalities regarding an involution $\theta$ of
a semisimple Lie algebra $\fg$ and orbits of $K=G^{\theta}$ on the
flag variety $\B$ of $\fg$ (see \cite{Mat79, RS, Vg,CEexp} for more details).
    Each Borel subalgebra $\fb$ of $\fg$
contains a $\theta$-stable Cartan subalgebra $\ft$.
Let $\Phi(\fg,\ft)$ denote the roots of $\ft$ in $\fg$,   let $\Phi_{\fb}^+ $
denote the roots of $\ft$ in $\fb$, which we take to be the positive roots.
Let $\fg_{\alpha}$ denote the root space for a root $\alpha$.
   Then the $\theta$-action on $\ft$
induces an action on $\Phi(\fg,\ft)$.   Using this action, we define the \emph{type} of a root $\alpha\in\Phi_{\fb}^+$ as follows.  A root $\alpha$ is called {\it real}
if $\theta(\alpha)=-\alpha$, {\it imaginary } if $\theta(\alpha)=\alpha$, 
and {\it complex } if $\theta(\alpha)\not= \pm \alpha.$   If $\alpha$ is
imaginary, then $\alpha$ is called {\it compact } if $\theta|_{\fg_\alpha}=\id$
and {\it noncompact} if $\theta|_{\fg_\alpha}=-\id.$    If $\alpha$ is complex, then $\alpha$ is called \emph{complex} $\theta$-{\it stable} if $\theta(\alpha)$
is positive, and otherwise is called \emph{complex} $\theta$-{\it unstable}.   
These notions do not depend on the choice of $\theta$-stable Cartan subalgebra 
$\ft\subset\fb$, nor on the choice of $\fb$ in the $K$-orbit $K\cdot \fb$.



\begin{exam}\label{ex:roottypes}
Let $\fg=\fso(2l+1,\C)$ and $\fk=\fso(2l,\C)$, and let $\theta=\Ad(t)$ be 
as in Section \ref{ss:symmetricreal}.  Let $\Phi(\fg, \fh)$ be the set 
of standard roots of $\fg$ as in (\ref{eq:sooddroots}).  The roots $\{\pm(\epsilon_{i}-\epsilon_{j}),\, \pm(\epsilon_{i}+\epsilon_{j}), \;1\leq i < j\leq l\}$ are compact imaginary, and the roots $\{ \pm \epsilon_{i}\; i=1,\dots, l\}$ are non-compact imaginary. 

Now let $\fg=\fso(2l,\C)$ and $\fk=\fso(2l-1,\C)$ and $\theta$ is as in Section \ref{ss:symmetricreal}.  
Then the simple roots $\alpha_{l-1}=\epsilon_{l-1}-\epsilon_{l}$ and $\alpha_{l}=\epsilon_{l-1}+\epsilon_{l}$ 
are complex $\theta$-stable with $\theta(\alpha_{l-1})=\alpha_{l}$.  Note that we can choose
as a representative for $\theta$ a nontrivial element in the Weyl group
of $GL(\C e_l + \C e_{-l})$.  
Therefore, the roots $\{\pm(\epsilon_{i}+\epsilon_{j}),\,\pm (\epsilon_{i}-\epsilon_{j}),\, 1\leq i<j\leq l-1\}$ are compact imaginary, whereas the roots
$\{\pm(\epsilon_{i}+\epsilon_{l}), \pm(\epsilon_{i}-\epsilon_{l}),\, 1\leq i\leq l-1\}$ are complex $\theta$-stable with $\theta(\epsilon_{i}\pm\epsilon_{l})=\epsilon_{i}\mp\epsilon_{l}.$  
The $\theta$-stable subspace
$\fg_{\alpha}\oplus\fg_{\theta(\alpha)}$ decomposes as $\fg_{\alpha}\oplus\fg_{\theta(\alpha)}=((\fg_{\alpha}\oplus\fg_{\theta(\alpha)})\cap \fk)\oplus((\fg_{\alpha}\oplus\fg_{\theta(\alpha)})\cap\fg^{-\theta}).$
\end{exam}

We make use of the following notation throughout the paper. 
\begin{nota}\label{nota:Weyl}

 Let $T$ be the maximal torus with Lie algebra $\ft$, and let $W$ be the Weyl group with respect to $T$. 
 For an element $w\in W$, let $\dot{w}\in N_{G}(T)$ be a representative of $w$.   
 If $\ft\subset\fb$, with $\fb\in\B$, then $\Ad(\dot{w})\fb$ is independent of the choice of representative $\dot{w}$ of $w$,
 and we denote it by $w(\fb)$.    
\end{nota}

Let $Q=K\cdot\fb$ and suppose that $\alpha\in\Phi_{\fb}^{+}$ is a simple root for $\fb$.  
  Let ${\Par}_\alpha$ be the variety of parabolic subalgebras
of type $\alpha$, and consider the projection $\pi_\alpha : \B \to {\Par}_\alpha.$  Let $m(s_\alpha)\cdot Q$ be the unique $K$-orbit
of maximal dimension in $\pi_\alpha^{-1}(\pi_\alpha(Q))$.   
For each simple root $\alpha$, choose root vectors $e_\alpha \in \fg_\alpha$,
$f_\alpha \in \fg_{-\alpha}$, and $h_\alpha=[e_\alpha, f_\alpha]$ such that $\mbox{span}\{e_{\alpha}, f_{\alpha}, h_{\alpha}\}$ forms a subalgebra 
of $\fg$ isomorphic to $\fsl(2,\C)$. 
Choose a Lie algebra homomorphism $\phi_\alpha:\fsl(2,\C) \to \fg$ such that:
\begin{equation}\label{eq:sl2map}
\phi_{\alpha}:\left[\begin{array}{cc} 0 & 1\\
0 & 0\end{array}\right]\to e_{\alpha},\;    
\phi_{\alpha}:\left[\begin{array}{cc} 0 & 0\\
1 & 0\end{array}\right]\to f_{\alpha},\;
\phi_{\alpha}:\left[\begin{array}{cc} 1 & 0\\
0 & -1\end{array}\right]\to h_{\alpha}\; 
\end{equation}
Also denote by $\phi_\alpha:SL(2,\C) \to G$  the induced Lie group homomorphism,
and let 
    \begin{equation}\label{eq:cayley}
u_{\alpha}=\phi_{\alpha}\left(\frac{1}{\sqrt{2}}\left[\begin{array}{cc} 1 & \imath\\
\imath & 1\end{array}\right]\right).
\end{equation}

\begin{lem}\label{lem_monoidaction} [\cite{RS}, 4.3]\\
Let $Q=K\cdot \fb$ be a $K$-orbit on $\B$.
\begin{enumerate}
\item  $m(s_\alpha)\cdot Q=Q$
unless $\alpha$ is either noncompact or complex $\theta$-stable, and when
$m(s_\alpha)\cdot Q \not=Q$, then $\dim(m(s_\alpha)\cdot Q)=\dim(Q)+1$.
\item  If $\alpha$ is noncompact for $Q$, then $m(s_\alpha)\cdot Q
=K\cdot \Ad(u_\alpha) \fb$ and the $K$-orbits in $\pi_\alpha^{-1}(\pi_\alpha(Q))$ are
$Q, m(s_\alpha)\cdot Q,$ and $K\cdot s_\alpha(\fb).$    
Further,
 $m(s_\alpha) \cdot K\cdot s_\alpha(\fb) = m(s_\alpha)\cdot Q.$    
\item If $\alpha$ is complex $\theta$-stable for $Q$, then $m(s_\alpha)\cdot \fb = K\cdot s_\alpha (\fb),$
and $\pi_\alpha^{-1}(\pi_\alpha(Q))$ consists of $Q$ and $m(s_\alpha)\cdot Q$.
\end{enumerate} 
\end{lem}

 The action by operators $m(s_\alpha)$ on $K$-orbits is called the monoidal
action \cite{RS}.  
\begin{lem}\label{l:allorbits}[\cite{RS}, Theorem 4.6]\\
Every $K$-orbit on $\B$ is of the form
$m(s_{\beta_1}) \cdots m(s_{\beta_k})\cdot \fb_1$, where $K\cdot\fb_1$ is a closed $K$-orbit
on $\B$, $k\ge 0$, and $\beta_1, \dots, \beta_k$ are simple roots. 
\end{lem}

We now briefly recall the classification of closed $K$-orbits 
on $\B$ from Section 4.3 of \cite{CEexp}.  Let $K\cdot \fb_0$
be a closed $K$-orbit with $\fb_0$ containing a $\theta$-stable Cartan subalgebra $\ft$
corresponding to a maximal torus $T$.  Since $T$ is $\theta$-stable,
 $\theta$ acts naturally on the Weyl group $W=N_G(T)/T$.  Further, the
subgroup $T\cap K$ is a Cartan subgroup of $K$  and the
Weyl group $W_K = N_K(T\cap K)/(T \cap K)$ embeds into $W$ (Lemmas 5.1 and 5.3, \cite{Rich}), and
is contained in $W^\theta,$ the fixed points of $\theta$ on $W$.

\begin{lem}\label{lem_orbitclass} [\cite{CEexp}, Theorem 4.10]
The map $W^\theta/W_K \to K\backslash \B$
given by $wW_K \mapsto K\cdot w^{-1}\cdot (\fb_0)$ is a bijection to the closed $K$-orbits
on $\B$.   
\end{lem}

\begin{rem}\label{r:isBorel}
If the 
$K$-orbit $Q=K\cdot\fb$ is closed, then $\fb\cap\fk$ is a Borel subalgebra in $\fk$ 
(see Lemma 5.1 of \cite{Rich}).  Thus, $\dim (Q)=\dim(\B_{\fk})$.
\end{rem}

We now return to the case where $\fg=\fso(n,\C)$ and $\fk=\fso(n-1,\C)$ and use 
Lemma \ref{lem_orbitclass} to determine the closed $K$-orbits on $\B$.  Before doing that, we state
a result on the relation between $W$ and $W_K$, which we will also need
later.    Recall that when $\fg=\fso(2l+1,\C)$, $W=S_l \rtimes U_l$, where
$U_l$ is the group generated by sign changes $\tau_i, 1\le i \le l$ in the root system, with
$\tau_i(\eps_i)=-\eps_i$ and $\tau_i(\eps_j)=\eps_j$ for $j\not= i.$   When
$\fg=\fso(2l,\C)$, $W=S_l \rtimes T_l$, where $T_l$ is the subgroup of
$U_l$ generated by products $\tau_i \tau_j, 1\le i < j \le l$.

\begin{prop}\label{prop_weylk}
\begin{enumerate}
\item Let $\fg=\fso(2l+1,\C)$ and let $\fk=\fso(2l,\C)$.
Then $W=W^{\theta}$ and $W/W_K = \{ eW_K, s_{\alpha_l}W_K \},$ where 
$e$ denotes the identity element in $W$.  In particular, $W_K$ has index $2$ in $W$.  Further, $W_K$ is the subgroup $S_l \rtimes T_l$ of $W$.
\item Let $\fg=\fso(2l,\C)$ and $\fk=\fso(2l-1,\C)$.    Then $W^\theta$ is
the subgroup of $W$ generated by the elements $s_{\alpha_1}, \dots, s_{\alpha_{l-2}},
s_{\alpha_{l-1}}\cdot s_{\alpha_l},$ and $W^\theta=W_K.$
\end{enumerate}
\end{prop}

\begin{proof}
For (1), since $\theta$ is inner, we know $W^{\theta}=W$,
 and the Cartan subalgebra $\fh$ of $\fg$ is also a Cartan subalgebra of $\fk$. 
By Example \ref{ex:roottypes}, the roots $\eps_{i}\pm \eps_{j}$ are exactly the roots of $\fk$ 
with respect to $\fh$.  The assertion about $W/W_{K}$ now follows by the above 
remarks on Weyl groups, and the observation that $s_{\alpha_{l}}=\tau_{l}\not\in W_K$. The rest of (1) follows easily.

For (2), the first statement follows from 1.32(b) in \cite{St}.   The second
statement can be deduced from the proof of (5) in Theorem 8.2 of \cite{St},
but we provide a more direct proof.
An easy calculation shows that for
$\sigma \in S_l$, $\theta \sigma \theta^{-1} = \sigma \cdot \tau_{\sigma^{-1}(l)} \cdot \tau_l.$ 
Note that $T_l$ is commutative and $\theta$ acts trivially on $T_l$.   It follows
that $W^\theta$ is identified with the semi-direct product of $S_{l-1}$ with
$T_l$.  Since $W_K \subset W^\theta$, the second statement follows.
\end{proof}  

Now it remains to describe the monoidal action.  
To determine the type of a root for a $K$-orbit $Q=K\cdot\fb$, 
it is convenient to replace the involution $\theta$ by another involution $\theta_{Q}$, which 
preserves the standard Cartan subalgebra of diagonal matrices $\fh$ and thus 
acts on the standard root system $\Phi(\fg,\fh)$.   Suppose that $\fb=\Ad(v)\fb_{+}$, 
where $\fb_{+}$ is the standard Borel subalgebra of upper triangular matrices.  Then 
$$
\theta_{Q}:=\Ad(v^{-1})\circ\theta\circ \Ad(v).
$$
It is easy to check  that the type 
of a standard positive root $\alpha\in \Phi^{+}(\fg,\fh)$ with respect to $\theta_{Q}$ 
is the same as the type of the positive root $\Ad(v)\alpha:=\alpha\circ\Ad(v^{-1})$ for $\fb$ with respect to $\theta$ 
(see Definition 4.6 and Proposition 4.7 of \cite{CEexp}).  In Section 4.4 of \cite{CEexp}, we give an 
inductive method of constructing the involution $\theta_{m(s_{\alpha})\cdot Q}$ from the involution $\theta_{Q}$ (Propositions 4.27 and 4.28).
 

\begin{prop} \label{prop_sooddflag}
Let $\fg=\fso(2l+1,\C)$ and $\fk=\fso(2l,\C)$.
\begin{enumerate}  
\item There are exactly  $l+2$ $K$-orbits on the flag variety $\B$ of $\fg$.
\item We let $\fb_+$ be the upper triangular matrices in $\fg$,
and let $\fb_-:=s_{\alpha_l}(\fb_+)$.    Exactly two $K$-orbits on $\B$ are
closed, and they are $Q_+:= K\cdot \fb_+$ and $Q_-:=K\cdot \fb_-.$
Further, $m(s_{\alpha_l})\cdot Q_+ = m(s_{\alpha_l})\cdot Q_-=K\cdot \Ad(u_{\alpha_l})\fb_+.$
\item The non-closed orbits are of the form $$Q_i:=m(s_{\alpha_{i+1}})\cdot m(s_{\alpha_{i+2}})
\cdots m(s_{\alpha_{l-1}}) \cdot m(s_{\alpha_l})\cdot Q_+$$ for $i=0, \dots, l-1$.
Moreover, the codimension of $Q_i$ in $\B$ is $i$.    Further, 
$$\fb_i:= \Ad(u_{\alpha_{l}} )s_{\alpha_{l-1}}  s_{\alpha_{l-2}} \dots s_{\alpha_{i+1}}  (\fb_+) \in Q_i.$$
In particular, the unique open $K$-orbit contains the Borel 
subalgebra 
\begin{equation}\label{eq:openborelodd}
\fb_{0}=\Ad(u_{\alpha_{l}})s_{\alpha_{l-1}}  s_{\alpha_{l-2}} \dots s_{\alpha_{1}}  (\fb_+).
\end{equation}
\end{enumerate}
\end{prop}

\begin{proof}
 For (2),  since $\fb_+$
is $\theta$-stable, $K\cdot \fb_+$ is closed by Proposition 4.12 of \cite{CEexp}.
   By Lemma \ref{lem_orbitclass} and part (1) of
Proposition \ref{prop_weylk}, 
there are two closed orbits, and they are
$\fb_+$ and $\fb_-$.   The assertion that $\fb_{l-1}=\Ad(u_{\alpha_l})\fb_+ \in
Q_{l-1}$ follows from part (2) of Lemma 
\ref{lem_monoidaction} and Example \ref{ex:roottypes}.  The second statement of part (2) of Lemma \ref{lem_monoidaction} implies
that $m(s_{\alpha_l})\cdot Q_{-}=m(s_{\alpha_l})\cdot Q_{+}=Q_{l-1}$.  
By Remark \ref{r:isBorel}, $\codim(Q_{+})=\codim(Q_{-})=\dim (\B)-\dim (\B_{\fk})=l$.  
Thus, $\codim(Q_{l-1})=l-1$ by part (1) of Lemma \ref{lem_monoidaction}.  

For (3), first recall that by Proposition 4.27 of \cite{CEexp}, the involution
$\theta_{Q_{l-1}}$ associated to $Q_{l-1}=K\cdot\fb_{l-1}$ is $\Ad(\dot{s}_{\alpha_l}^{-1}) \circ \theta$, so the
involution on the standard root system is $s_{\alpha_l}$, which is a sign change in
the last variable.   Since $\theta=\Ad(t)$, it follows from a calculation in $SO(3,\C)$
 that we can choose
the representative ${\dot{s}}_{\alpha_l}$ so that $\theta_{Q_{l-1}} \in GL(\C e_l + \C e_{-l})$.
It follows easily that $\alpha_1, \dots, \alpha_{l-2}$ are compact for $\theta_{Q_{l-1}}$,
while $\alpha_{l-1}$ is complex $\theta_{Q_{l-1}}$-stable, and $\alpha_l$ is real.  Hence, the only monoidal action which gives us a new orbit is 
$m(s_{\alpha_{l-1}})\cdot Q_{l-1} = Q_{l-2}$.    By part (3) of Lemma \ref{lem_monoidaction}, the
Borel subalgebra $\fb_{l-2} = \Ad(\tilde{\dot{s}}_{\alpha_{l-1}})(\fb_{l-1}),$ where
$\tilde{\dot{s}}_{\alpha_{l-1}}$ is a representative of the simple reflection
$s_{\alpha_{l-1}}$ defined with respect to $\fb_{l-1}$.   Since
$\fb_{l-1}=\Ad(u_{\alpha_l})\fb_+$, it follows that $\tilde{\dot{s}}_{\alpha_{l-1}}
=u_{\alpha_l} \dot{s}_{\alpha_{l-1}} u_{\alpha_l}^{-1}$.  Hence,
$\fb_{l-2} = \Ad(u_{\alpha_l} \dot{s}_{\alpha_l} u_{\alpha_l}^{-1} u_{\alpha_l})(\fb_+),$
which verifies the last part of (3) for $i=l-2$.
By Proposition 4.28 of \cite{CEexp}, the involution $\theta_{Q_{l-2}}$ associated
to $Q_{l-2}$ is 
$\Ad({\dot{s}}_{\alpha_{l-1}})^{-1} \circ \theta_{Q_{l-1}} \circ \Ad({\dot{s}}_{\alpha_{l-1}}).$
We can choose ${\dot{s}}_{\alpha_{l-1}}$, so that 
$\theta_{Q_{l-2}} \in GL(\C e_{l-1} + \C e_{-(l-1)}).$   Thus, $\alpha_1, \dots, \alpha_{l-3},$
and $\alpha_l$ are compact, while $\alpha_{l-2}$ is complex $\theta_{Q_{l-2}}$-stable
and $\alpha_{l-1}$ is complex $\theta_{Q_{l-2}}$-unstable.
Now an inductive argument, which we leave to the reader, shows that if we 
define $\fb_i$ as in assertion (3), and let $\theta_{Q_{i}}$ be the involution relative to $Q_{i}$,
 the roots $\alpha_1, \dots, \alpha_{i-1}$ are compact
for $\theta_{Q_{i}}$, $\alpha_{i}$ is $\theta_{Q_{i}}$-stable, $\alpha_{i+1}$ is $\theta_{Q_{i}}$-unstable,
and $\alpha_{i+2}, \dots, \alpha_l$ are compact for $\theta_{Q_{i}}$.  Hence, from $Q_i$,
the only monoidal action which gives a new orbit is $m(s_{\alpha_{i}})\cdot Q_i = Q_{i-1}$
and $Q_{i-1}=K\cdot \fb_{i-1},$ by using the same argument as in the case
$i=l-2$.  As a consequence, the codimension of $Q_{i-1}$ in $\B$ is $i-1$.
It now follows that $Q_0, \dots, Q_{l-1}$ are distinct orbits.   The induction
argument implies that no monoidal actions change $Q_0$, and it follows by
Lemma \ref{l:allorbits} that $Q_+, Q_-, Q_{l-1}, \dots, Q_0$ are all the
$K$-orbits.
   This completes the proof of (3), and (1) is an easy consequence.
\end{proof}

\begin{prop}\label{prop_soevenflag} 
Let $\fg=\fso(2l,\C)$ and $\fk=\fso(2l-1,\C)$.
\begin{enumerate}
\item There are exactly $l$ $K$-orbits in the flag variety $\B$ of $\fg$.
\item Let $\fb_+$ be the set of upper triangular matrices in $\fg$.
Then $Q_+:= K\cdot \fb_+$ is the only closed $K$-orbit.
\item Let $$Q_i:=m(s_{\alpha_i}) \dots m(s_{\alpha_{l-1}}) \cdot Q_+$$
and let $$\fb_i:= s_{\alpha_{l-1}} s_{\alpha_{l-2}} \dots  s_{\alpha_{i}}(\fb_+)\mbox{ for } i=1,\dots, l-1$$
Then $Q_i=K\cdot \fb_i$ has codimension $i-1$ in $\B$.   The distinct $K$-orbits are 
$Q_+, Q_{l-1}, \dots, Q_1$.
 In particular, the unique open orbit is $Q_{1}$ and contains the Borel subalgebra 
\begin{equation}\label{eq:openboreleven}
\fb_{1}=s_{\alpha_{l-1}} s_{\alpha_{l-2}} \dots  s_{\alpha_{1}}(\fb_+).
\end{equation}
\end{enumerate}
\end{prop}

\begin{proof} For (2), since $\fb_+$ is preserved by $\theta$, $Q_{+}=K\cdot \fb_+$
is closed by Proposition 4.12 of \cite{CEexp}.  
  Thus, $Q_+$ is the unique closed $K$-orbit by part (2) of 
Proposition \ref{prop_weylk} and Lemma \ref{lem_orbitclass}.  
By Remark \ref{r:isBorel}, we have $\codim(Q_{+})=\dim(\B)-\dim(\B_{\fk})=l-1$. 

For (3), we saw in Example \ref{ex:roottypes} that $\alpha_{l-1}$ and $\alpha_{l}$ are complex $\theta$-stable, and that all other simple roots
are compact. Hence, $m(s_{\alpha_l})\cdot Q_+$ and $m(s_{\alpha_{l-1}})\cdot Q_+$
are the orbits of dimension $\dim Q_{+}+1$.  We claim that they coincide.  Indeed, by part (3) of Lemma \ref{lem_monoidaction}, $m(s_{\alpha_l})\cdot Q_+=K\cdot s_{\alpha_l}(\fb_+)$ and
$m(s_{\alpha_{l-1}})\cdot Q_+=K\cdot s_{\alpha_{l-1}}(\fb_+)$.   We may choose
the representatives for $s_{\alpha_{l-1}}$ and $s_{\alpha_l}$ in $W$ so that
$\theta({\dot{s}}_{\alpha_{l-1}})={\dot{s}}_{\alpha_l}$, and since $\alpha_{l-1}$ and
$\alpha_{l}$ are perpendicular, we may assume ${\dot{s}}_{\alpha_{l-1}}$ and
${\dot{s}}_{\alpha_l}$ commute.
Note that
$s_{\alpha_l}(\fb_+)=s_{\alpha_l}s_{\alpha_{l-1}}s_{\alpha_{l-1}}(\fb_+).$   It follows that
$$
\theta({\dot{s}}_{\alpha_l}{\dot{s}}_{\alpha_{l-1}})={\dot{s}}_{\alpha_{l-1}}{\dot{s}}_{\alpha_l}=
{\dot{s}}_{\alpha_l}{\dot{s}}_{\alpha_{l-1}}.
$$
Thus, $\dot{s}_{\alpha_l} \dot{s}_{\alpha_{l-1}}
\in K$, and hence $K\cdot s_{\alpha_{l-1}}(\fb_+) = K\cdot  s_{\alpha_{l}}(\fb_+),$
which establishes the claim.
The orbit $Q_{l-1} = m(s_{\alpha_{l-1}})\cdot Q_+$ has involution $\theta_{Q_{l-1}} = 
\Ad(\dot{s}_{\alpha_{l-1}}^{-1}) \circ \theta  \circ \Ad(\dot{s}_{\alpha_{l-1}}),$
which changes the sign of the $l-1$ coordinate of $\fh$, and no other coordinates.
It now follows that $\alpha_{l-2}$ is complex $\theta_{Q_{l-1}}$-stable, while
$\alpha_1, \dots, \alpha_{l-3}$ are compact, and $\alpha_{l-1}$ and $\alpha_l$
are complex $\theta_{Q_{l-1}}$-unstable.   The remainder of the argument follows
by an easy induction similar to the proof of part (3) of Proposition
\ref{prop_sooddflag}.  Part (1) is an easy consequence of (2) and (3).
\end{proof}


\section{Orthogonal Eigenvalue coincidence varieties and $K$-orbits}\label{ss:irredcomp}

We prove Theorem \ref{thm:bigthm} in this section.


\subsection{The varieties $\overline{Y_{Q}}$ as irreducible components of $\fg(\geq i)$}
Using our work in Sections \ref{ss:YQ} and \ref{ss:Korbits}, we show that
 the Zariski closure of the varieties $\overline{Y_{Q}}$ with $\codim(Q)=i$ are irreducible components 
 of the eigenvalue coincidence varieties $\fg(\geq i)$. We consider the case where $\fg$ is type $D$ and type $B$ separately. 
 We first consider $Y_{Q}$, where the $K$-orbit $Q$ not closed.

Case I: $\fg=\fso(2l+1,\C)$, $\fk=\fso(2l,\C)$
\begin{thm}\label{thm:YQs2}
Let $\fg(\geq i)$, $i=0, \dots, l-1$ be the orthogonal eigenvalue coincidence variety 
defined in (\ref{eq:gidfn}).  Let $Q=K\cdot \fb\subset \B$ be a $K$-orbit 
with $\mbox{codim}(Q)=i$.
\begin{enumerate}
\item There exists a $\theta$-stable parabolic subalgebra $\fr$ with $\fb\subset \fr$ such that the 
hypothesis of Proposition \ref{prop:dimYQs} is satisfied.  The parabolic subalgebra $\fr$ has $\theta$-stable Levi 
decomposition
\begin{equation}\label{eq:Levi2}
\fr=\fl\oplus\fu\mbox{ with } \fl_{ss}\cong\fso(2(l-i)+1,\C) \mbox{ and }\fz\cong\ (\fgl(1,\C))^{i}.
\end{equation}
Let $L_{ss}\cong SO(2(l-i)+1,\C)\subset G$ be the connected algebraic subgroup with Lie 
algebra $\fl_{ss}$.  The restriction $\theta|_{\fl_{ss}}= \theta_{2(l-i)+1}$ is the 
involution on $\fso(2(l-i)+1,\C)$ defining $\fso(2(l-i), \C)$, so that
\begin{equation}\label{eq:Kfixed}
 \fl_{ss}^{\theta}=\fl_{ss}\cap\fk\cong \fso(2(l-i),\C) \mbox{ and thus }  (L_{ss}^{\theta})^{0}=K\cap L_{ss}\cong SO(2(l-i),\C).
\end{equation}
Furthermore, $SO(2(l-i),\C)\cdot(\fb\cap\fl_{ss})$ is open in $\B_{\fso(2(l-i)+1,\C)}$.  

\item We have $\overline{Y_{Q}}=Y_{Q_{\fr}}$, and the variety $Y_{Q_{\fr}}$ is an 
irreducible component of $\fg(\geq i)$.  
\end{enumerate}

\end{thm}
\begin{proof}
We first prove (1).    By $K$-equivariance, it suffices
to prove the statement for any representative $\fb$ of the $K$-orbit $Q$ of codimension
$i$.  By part (3) of Proposition \ref{prop_sooddflag}, we can take 
$\fb=\fb_{i}=\Ad(u_{\alpha_{l}})s_{\alpha_{l-1}}\dots s_{\alpha_{i+1}}(\fb_{+})$.  
Let $\fr\subset \fg$ be the standard parabolic subalgebra generated by $\fb_{+}$ and the
negative simple root spaces $\fg_{-\alpha_{l}}, \fg_{-\alpha_{l-1}},\dots, \fg_{-\alpha_{i+1}}.$ 
Note that $\fr$ is $\theta$-stable with Levi decomposition (\ref{eq:Levi2}) and  also $\theta|_{\fl_{ss}}=\theta_{2(l-i)+1}$.  Equation (\ref{eq:Kfixed}) follows.  To see that $\fb_i\subset \fr$, note that we can choose the representative $\dot{s}_{\alpha_{j}}$ of $s_{\alpha_{j}}$ so that $\dot{s}_{\alpha_{j}}\in L_{ss}$ for $j=i+1,\dots, l$, and 
$u_{\alpha_{l}}\in L_{ss}$ by Equation (\ref{eq:cayley}).  Thus, the element 
\begin{equation}\label{eq:v}
v:=u_{\alpha_{l}}\dot{s}_{\alpha_{l-1}}\dots \dot{s}_{\alpha_{i+1}}\in L_{ss}\subset R.  
\end{equation}
Hence, $\fb_{i}=\Ad(v)\fb_{+}\subset \Ad(v)\fr=\fr$.

It remains to show that $(K\cap L_{ss})\cdot (\fb\cap\fl_{ss})$ can be identified with the open 
$SO(2(l-i),\C)$-orbit in the flag variety $\B_{\fso(2(l-i)+1,\C)}$.  
Note that $\fb_{+}\cap \fl_{ss}$ can be identified with the standard Borel 
subalgebra $\fb_{+,\fso(2(l-i)+1,\C)}$ of upper triangular matrices 
in $\fso(2(l-i)+1,\C)$.  Since the element $v$ in Equation (\ref{eq:v}) is in $L_{ss}$, we have:
$$
\fb\cap\fl_{ss}=(\Ad(v)\fb_{+})\cap\fl_{ss}=\Ad(v)(\fb_{+}\cap\fl_{ss})=\Ad(v)\fb_{+,\fso(2(l-i)+1,\C)}.  
$$
It follows from Equations (\ref{eq:openborelodd}) and (\ref{eq:v})
 that $\Ad(v)\fb_{+,\fso(2(l-i)+1,\C)}\subset \B_{\fso(2(l-i)+1,\C)}$ is a representative 
of the open $SO(2(l-i),\C)$-orbit on $\fb_{\fso(2(l-i)+1,\C)}$. 

We now prove (2).  The first statement of (2) follows 
immediately from part (1) and Proposition \ref{prop:dimYQs}.  
By Proposition \ref{prop:dimYQ}, to see that $Y_{\fr}$ is an irreducible component of $\fg(\geq i)$, it suffices 
to show that $Y_{\fr}\subset \fg(\geq i)$.  
Consider the partial Kostant-Wallach map $\Phi_{n}$ defined in Equation (\ref{eq:partial}).  
Let $\fq$ be a parabolic subalgebra of $\fg$ with $\fq\in Q_{\fr}$, and let 
$y\in \fq$.  We need to show that $\Phi_{n}(y)\in V^{r_{n-1}, r_{n}}(\geq i)$.  
Since the map $\Phi_{n}$ is $K$-invariant, it is enough to show that 
$\Phi_{n}(x)\in V^{r_{n-1}, r_{n}}(\geq i)$ for $x\in\fr$.  
Recall that $\Phi_{n}(x)=(\chi_{n-1}(x_{\fk}), \chi_{n}(x))$, where 
$\chi_{i}:\fso(i,\C)\to \fso(i,\C)//SO(i,\C)$ is the adjoint quotient. 
For $x\in\fr$, let $x_{\fl}$ be the projection of $x$ onto $\fl$ off of $\fu$.  
It is well-known that $\chi_{n}(x)=\chi_{n}(x_{\fl})$.  Using the decomposition in 
(\ref{eq:Levi2}), we can write $x_{\fl}$ as $x_{\fl}=x_{\fz}\oplus x_{\fl_{ss}}$ with
$x_{\fz}\in\fz\cong (\fgl(1,\C))^{i}$ and $x_{\fl_{ss}}\in\fl_{ss}=\fso(2(l-i)+1,\C)$.   
It is easy to see that the coordinates of $x_{\fz}$ are in the spectrum of $x$. 
Since $\fr$ is $\theta$-stable, $\fk\cap\fr$ is a parabolic subalgebra of $\fk$ with 
Levi decomposition: $$\fk\cap\fr=\fk\cap\fl\oplus\fk\cap\fu \mbox{ and }
\fk\cap\fl=\fk\cap\fz\oplus\fk\cap\fl_{ss}\cong\fz\oplus \fso(2(l-i),\C), $$
where the isomorphism follows from (\ref{eq:Kfixed}) and the observation that $\fz\subset\fh\subset\fk$ (see Section \ref{ss:symmetricreal}).
Since $x_{\fk}\in\fk\cap\fr$, we know $\chi_{n-1}(x_{\fk})=\chi_{n-1}((x_{\fk})_{\fl\cap\fk}).$  
Now, $(x_{\fk})_{\fl\cap\fk}=x_{\fz}+x_{\fso(2(l-i),\C)}$, and the coordinates of $x_{\fz}$ are in the spectrum of 
$(x_{\fk})_{\fl\cap\fk}$.  Thus, Remark \ref{r:thesame} implies that $\Phi_{n}(x)=(\chi_{n-1}(x_{\fk}), \chi_{n}(x))\in V^{r_{n-1}, r_{n}}(\geq i)$, and it follows 
that $Y_{\fr}\subset \fg(\geq i)$.  

\end{proof}

Case II: $\fg=\fso(2l,\C)$, $\fk=\fso(2l-1,\C)$.\\

\begin{thm}\label{thm:YQs1}
Let $\fg(\geq i-1)$ for $i=1, \dots, l-1$ be the orthogonal eigenvalue coincidence variety 
defined in (\ref{eq:gidfn}).  Let $Q=K\cdot \fb\subset \B$ be a $K$-orbit 
with  $\codim(Q)=i-1$. 
\begin{enumerate}
\item  There exists a $\theta$-stable parabolic subalgebra $\fr$ with $\fb\subset \fr$, and $\fr$ satisfies
the hypothesis of Proposition \ref{prop:dimYQs}.  The parabolic subalgebra $\fr$ has $\theta$-stable Levi decomposition 
\begin{equation}\label{eq:Levi1}
\fr=\fl\oplus\fu\mbox{ with } \fl_{ss}\cong\fso(2(l-i)+2,\C) \mbox{ and }\fz\cong(\fgl(1,\C))^{i-1}.
\end{equation}
Let $L_{ss}\cong SO(2(l-i)+2,\C)\subset G$ be the connected algebraic subgroup with Lie 
algebra $\fl_{ss}$.  Then $\theta |_{\fl_{ss}}=\theta_{2l-2i+2}$, so that 
$$
 \fl_{ss}^{\theta}=\fl_{ss}\cap\fk\cong\fso(2(l-i)+1,\C) \mbox{ and thus }  (L_{ss}^{\theta})^{0}=K\cap L_{ss}\cong SO(2(l-i)+1,\C).
$$
Furthermore, $SO(2(l-i)+1,\C)\cdot (\fb\cap\fl_{ss})$ is open in $\B_{\fso(2(l-i)+2,\C)}$.
\item We have $\overline{Y_{Q}}=Y_{Q_{\fr}}$, and the variety $Y_{Q_{\fr}}$ is an 
irreducible component of $\fg(\geq i-1)$.  
\end{enumerate}

\end{thm}

\begin{proof}
The proof is very similar to the proof of Theorem \ref{thm:YQs2}.  
We begin with the proof of part (1).  
  Again, by $K$-equivariance, it suffices to prove the statement 
for any representative $\fb$ of the $K$-orbit $Q$ of codimension $i-1$.  By part (3) of Proposition \ref{prop_soevenflag}, we can 
take $\fb=\fb_{i}=s_{\alpha_{l-1}}s_{\alpha_{l-2}}\dots s_{\alpha_{i}}(\fb_{+})$. 
Let $\fr\subset \fg$ be the standard parabolic subalgebra generated by $\fb_{+}$ and the
negative simple root spaces $\fg_{-\alpha_{l}}, \fg_{-\alpha_{l-1}},\dots, \fg_{-\alpha_{i}}.$  
We claim that $\fr$ is $\theta$-stable.  Indeed, we saw in Example \ref{ex:roottypes} that the roots
$\alpha_{i}$ are compact imaginary for $i=1,\dots, l-2$ and that $\alpha_{l-1}$ and $\alpha_{l}$
are complex $\theta$-stable with $\theta(\alpha_{l-1})=\alpha_{l}$.  It then follows easily that $\fr$ has Levi decomposition (\ref{eq:Levi1}) and that 
$\theta|_{\fl_{ss}}=\theta_{2(l-i)+2}$, whence $\fl_{ss}^{\theta}=\fk\cap\fl_{ss}\cong \fso(2(l-i)+1,\C)$, and
$(L_{ss}^{\theta})^{0}=K\cap L_{ss}\cong SO(2(l-i)+1,\C)$.  The remainder of the proof 
proceeds exactly as in the proof of part (1) of Theorem \ref{thm:YQs2}, using Equation (\ref{eq:openboreleven}) instead of Equation 
(\ref{eq:openborelodd}).   

The proof of (2) is also analogous to the proof of part (2) of Theorem \ref{thm:YQs2}.  
 The key observation is that for $x\in\fr$ with $x_{\fl}=x_{\fz}\oplus x_{\fl_{ss}}$ the coordinates of $x_{\fz}\in\fz\cong(\fgl(1,\C))^{i-1}$ are in the spectrum of both $x\in\fg$ and $x_{\fk}\in\fk$.  To show this, one observes that $\fz \subset \fk$, which follows since $\theta$ permutes the simple roots of $\fl$.   We leave the remaining
details to the reader.
\end{proof}

\begin{rem}\label{r:centerink}
Note that $\fz \subset \fk$, where $\fz$ is the centre of the Levi subalgebras $\fl$ in 
Theorems \ref{thm:YQs2} and \ref{thm:YQs1}.
\end{rem}

We now consider the case where $Q$ is a closed $K$-orbit. 
\begin{thm}\label{thm:YQclosed}
Let $Q$ be a closed $K$-orbit on $\B$.  Then $Y_{Q}$ is an irreducible component of $\fg(\geq r_{n-1})=\fg(r_{n-1})$.  
\end{thm}
\begin{proof}
We show that for a closed $K$-orbit $Q=K\cdot\fb$, the subvariety $Y_{Q}\subset \fg(r_{n-1})$.  It then 
follows from Proposition \ref{prop:dimYQ} that $Y_{Q}$ is an irreducible component of $\fg(r_{n-1})$.  
By $K$-equivariance, it suffices to show that $\fb\subset\fg( r_{n-1})$.  
If $\fg$ is of type $B$, then part (2) of Proposition \ref{prop_sooddflag} implies that 
$\fb=\fb_{+}$ or $\fb=s_{\alpha_{l}}(\fb_{+})$.  In either case, $\fb$ contains 
the standard diagonal Cartan subalgebra $\fh$ of $\fg$.  
Now by Remark \ref{r:isBorel}, $\fb\cap\fk$ is a Borel subalgebra of $\fk$ with Levi decomposition 
$$
\fb\cap\fk=\fh\oplus(\fn\cap\fk), 
$$
where $\fn=[\fb,\fb]$ is the nilradical of $\fb$.  Thus, 
for $x\in\fb$ with $x=x_{\fh}+ x_{\fn}$, with $x_{\fh}\in\fh$ and $x_{\fn}\in\fn$, 
the coordinates of $x_{\fh}$ are in the spectrum of both $x$ and $x_{\fk}$.  
It follows that $\fb\subset\fg(r_{n-1})$.  

If $\fg$ is of type $D$, then part (2) of Proposition \ref{prop_soevenflag} 
states that $Q_{+}=K\cdot\fb_{+}$ is the only closed $K$-orbit.  We recall that 
$\theta(\epsilon_{r_{n}})=-\epsilon_{r_{n}}$, and $\theta(\epsilon_{i})=\epsilon_{i}$ for all $i\neq r_{n}$ 
(see Section \ref{ss:symmetricreal}).  Therefore, $\fh\cap\fk=\mbox{diag}[b_{1},\dots, b_{r_{n-1}},0,0, -b_{r_{n-1}}, \dots, -b_{1}]$, 
and $\fb\cap \fk$ is a Borel subalgebra of $\fk$ with Levi decomposition $\fb\cap\fk=\fh\cap\fk\oplus\fn\cap\fk$.  
Thus, for any $x\in\fb$, $x=x_{\fh}+x_{\fn}$, and $x_{\fh}=x_{\fh\cap \fk}+x_{\fh\cap\fg^{-\theta}}$, with 
$x_{\fk}=x_{\fh\cap\fk}+x_{\fn\cap\fk}\in\fk\cap\fb.$  Thus, $x_{\fh\cap\fk}\in \fh$ is in the spectrum of both $x$ and $x_{\fk}$.  
It follows that $\fb_{+}\subset\fg(r_{n-1})$.

\end{proof}

\begin{rem}\label{r:firstspremark}
As we noted in Remark \ref{r:itholds}, the hypothesis of Proposition \ref{prop:dimYQs} is true for the real rank one symmetric pair $(\fg=\fgl(n,\C), \fk=\fgl(n-1,\C)\oplus\fgl(1,\C))$ and for any $K$-orbit $Q$.  The analogue of part (2) of Theorems 
\ref{thm:YQs2} and \ref{thm:YQs1} also holds in this setting (see Theorems 3.6 and 3.7, \cite{CEeigen}).

Let $(\fg, \fk)$ be a symmetric pair, and let 
$Q=K\cdot\fb\subset \B$ be an arbitrary $K$-orbit in the flag variety $\B$ of $\fg$.  Then if 
$\fr\subset\fg$ is a $\theta$-stable parabolic subalgebra with $\fb\subset\fr$, the $K$-orbit $Q$ has
the structure of a fibre bundle as in (\ref{eq:bundle}).  However, $(K\cap L_{ss})\cdot (\fb\cap\fl_{ss})$ need not be the open $K$-orbit in $\B_{\fl_{ss}}$.

The hypothesis of Proposition \ref{prop:dimYQs} does not hold for
 the real rank one symmetric pair with $\fg=\mathfrak{sp}(2n,\C)$ and $\fk=\fsp(2n-2,\C)\oplus\fsp(2, \C).$  Further, one
can show that for this case,
 the varieties $Y_Q$ are not irreducible components of
the natural 
eigenvalue coincidence varieties.   It would be interesting to further
analyze objects analogous to those studied in this paper in that example.
\end{rem}


\subsection{Every irreducible component of $\fg(\geq i)$ is of the form $\overline{Y_{Q}}$}
In this section, we complete the last step of the proof of Theorem \ref{thm:bigthm}.  Consider the regular semisimple elements $\fk^{rs}$ of $\fk$, and let 
$\fh_{\fk}^{reg}=\fk^{rs}\cap\fh_{\fk}$.  For $x$ in $\fg$, consider the spectrum $\sigma(x_{\fk})=\{\pm a_{1},\dots, \pm a_{r_{n-1}}\}$ of $x_{\fk}$.  If $\fk$ is type $D$, $x_{\fk} \in \fk^{rs}$ if and only if $a_i \neq \pm a_j$ for $i\neq j$.
If $\fk$ is type $B$, $x_{\fk}\in \fk^{rs}$ if and only if $a_i  \neq \pm a_j$ for
$i\neq j$, and all $a_i \neq 0.$






\begin{thm}\label{thm:exhaustion}
Every irreducible component of the variety $\fg(\geq i)$, $i=0,\dots, r_{n-1}$ is of the form
$\overline{Y_{Q}}$ for some $K$-orbit $Q$ on $\B$ with $\mbox{codim}(Q)=i$.  
\end{thm}
\begin{proof}

Consider the set:
\begin{equation}\label{eq:U}
U:=\{x\in\fg :\; x_{\fk}\in\fk^{rs} \mbox{ and } 0\notin \sigma(x_{\fk})\}.
\end{equation}
Let $U(\geq i):=U\cap \fg(\geq i)$.  Note that $\fh\cap U(\geq i)\neq \emptyset$, so that 
$U$ and $U(\geq i)$ are non-empty Zariski open subsets of $\fg$ and $\fg(\geq i)$ respectively.  
By Proposition \ref{prop:flat}
and Exercise III.9.1 of \cite{Ha}, $\Phi_{n}(U)\subset V^{r_{n-1}, r_{n}}$ is open. Thus, 
$V^{r_{n-1}, r_{n}}(\geq i)\setminus \Phi_{n}(U)$ is a proper, closed subvariety of $V^{r_{n-1},r_{n}}(\geq i)$
and therefore has positive codimension by Lemma \ref{lem:vnirreducible}.  
It follows by Propositions \ref{prop:flat} and \ref{prop_dimgl} and Corollary III.9.6 of \cite{Ha} that 
$\fg(\geq i)\setminus U(\geq i)=\Phi^{-1}_{n}(V^{r_{n-1},r_{n}}(\geq i)\setminus \Phi_{n}(U))$ 
is a proper, closed subvariety of $\fg(\geq i)$ of positive codimension.  Since $\fg(\geq i)$ is equidimensional, 
it follows that $Z\cap U(\geq i)$ is nonempty for any irreducible component $Z$ of $\fg(\geq i)$. 
Thus, it suffices to show that 
\begin{equation}\label{eq:tricky}
U(\geq i)\subset\bigcup_{\codim(Q)=i}\overline{Y_{Q}}.
\end{equation}  

To prove Equation (\ref{eq:tricky}), we consider the following subvariety of $U(\geq i)$:
\begin{equation}\label{eq:XI}
\Xi:=\{x\in U(\geq i):\; x_{\fk}=(a_{1},\dots, a_{r_{n-1}})\in\fh_{\fk}^{reg}, \mbox{ and } \sigma(x_{\fk})\cap\sigma(x)\supset\{\pm a_{1},\dots, \pm a_{i}\},\, a_{j}\neq 0\, \forall j\}. 
\end{equation}
It is easy to check that any element of  $U(\geq i)$ is $K$-conjugate
to an element in $\Xi$.  Thus, by the $K$-equivariance of the varieties $\overline{Y_{Q}}$, it is enough to show that 
\begin{equation}\label{eq:Xicontainment}
\Xi\subset\displaystyle\bigcup_{\mbox{codim}(Q)=i} \overline{Y_{Q}}. 
\end{equation}

We consider the cases where $\fg$ is type $B$ and type $D$ separately.  First, we assume that $\fg=\fso(2l+1,\C)$.   By Theorem \ref{thm:YQs2}, it suffices
to show that 
\begin{equation}\label{eq:ilessl}
\Xi\subset Y_{Q_{\fr}} \mbox{ for } i<l, 
\end{equation}
where $\fr$ is the parabolic subalgebra generated by $\fb_{+}$ and the negative 
simple root spaces $\fg_{-\alpha_{i+1}},\dots, \fg_{-\alpha_{l}}$.  For $i=l$ we need to show that 
\begin{equation}\label{eq:iisl} 
\Xi\subset Y_{Q_{+}}\cup Y_{Q_{-}}, 
\end{equation}
where $Q_{+}=K\cdot\fb_{+}$ and $Q_{-}=K\cdot\fb_{-}$ are the distinct closed $K$-orbits on $\B$ (see part (2) of Proposition \ref{prop_sooddflag}). 
To prove Equations (\ref{eq:ilessl}) and (\ref{eq:iisl}), we need to describe
the variety $\Xi$ in more detail.  
Recall from Example \ref{ex:roottypes} that $$\fg^{-\theta}=\bigoplus_{j=1}^{l} \fg_{ \epsilon_{j}}\oplus \fg_{-\eps_{j}}.$$ 
Let $e_{\pm \eps_{j}}\in\fg_{\pm\eps_{j}}$ be a nonzero root vector.  
Consider elements of the form: 
\begin{equation}\label{eq:Xirootspace}
\underline{a}\displaystyle\oplus_{j=1}^{l} u_{j} e_{\epsilon_{j}}\oplus_{j=1}^{l} v_{j}e_{-\epsilon_{j}}, 
\end{equation}
where $\underline{a}=\mbox{diag}[a_{1},\dots, a_{l}, 0, -a_{l}, \dots, -a_{1}]\in\fh, \, a_{i}\neq \pm a_{j} \mbox{ if } i\not= j, \mbox{ and each } a_{i}\neq 0$.  
We choose the root vectors $e_{\pm \eps_{j}}$ so that 
$\Xi$ consists of matrices of the form:  
\begin{equation}\label{eq:BXimatrix}
X:=\left[\begin{array}{ccccccc}
a_{1}  &  \dots & 0  & u_{1} & 0 & \dots & 0\\
\vdots & \ddots & \vdots &\vdots & \vdots & & \vdots\\
0&\dots & a_{l} & u_{l} & 0 & \dots & 0\\
v_{1}& \dots & v_{l} & 0 & -u_{l} & \dots & -u_{1}\\
0 & \dots & 0 &-v_{l}& -a_{l} & \dots &  0\\
\vdots &  & \vdots & \vdots & \vdots& \ddots & \vdots \\
0 & \dots & 0&-v_{1} & 0 &\dots & -a_{1}\end{array}\right]
\end{equation}
with $a_k\not=\pm a_j$ for $k\not= j$, $a_j\not= 0$ for $j=1, \dots, l$,
and $\pm a_j$ is an eigenvalue of $X$ for $j=1, \dots, i$.
It is easy to see that the elements $\pm a_j$ for $j=1, \dots, i$
are eigenvalues of $X$ if and only if
\begin{equation}\label{eq:overlapconds}
u_{j}v_{j}=0  \mbox{ for } j=1,\dots, i.
\end{equation}
This follows easily from the fact that $a_{j}$ is an eigenvalue of $X$ if and only if the matrix
$X-a_{j} Id_{2l+1}$ is singular, where $Id_{2l+1}$ denotes the $(2l+1)\times (2l+1)$ identity matrix.

We can now describe the irreducible components of $\Xi$ using (\ref{eq:overlapconds}).  
For $k=1,\dots, i$, we define an index $j_{k}$ equal to either $j_{k}=U$ ($U$ for upper) or $j_{k}=L$ ($L$ for lower).  Consider the subvariety $\Xi_{j_{1},\dots, j_{i}}\subset \Xi$ defined by:
 \begin{equation}\label{eq:UorL}
\Xi_{j_{1},\dots, j_{i}}:=\{ x\in\Xi: v_{k}=0\mbox{ if } j_{k}=U, u_{k}=0 \mbox{ if } j_{k}=L\}.
\end{equation}
    Then
      \begin{equation}\label{eq:upsunion}
  \Xi= \bigcup_{j_{k}=U,\, L} \Xi_{j_{1},\dots, j_{i}}
  \end{equation} 
  is the irreducible component decomposition of $\Xi$.  
  Notice that in the case $j_{k}=U$ for all $k=1,\dots, i$, then 
  \begin{equation}\label{eq:isinfr}
  \Xi_{U,\dots, U}\subset \fr.
  \end{equation}
     This follows from the observation that 
  $\epsilon_{j}=\alpha_{j}+\dots + \alpha_{l}$ for any $j=1,\dots, l$.  Thus, for $j=i+1,\dots, l$, 
  $\fg_{\pm\epsilon_{j}}\subset \fl_{ss}\subset \fr$, where $\fl_{ss}$ is the semisimple 
  part of the Levi factor of $\fr$, and $\fg_{\epsilon_{j}}\subset\fu$ for $j=1,\dots, i$ (see (\ref{eq:Levi2})).  Observe also that if $j_{k}=L$ for 
  some $k=1,\dots, i$, then 
  \begin{equation}\label{eq:LtoU}
  \Ad(\dot{s}_{\epsilon_{k}}) \Xi_{j_{1},\dots, j_{k-1}, L, \dots, j_{i}}=\Xi_{j_{1},\dots, j_{k-1}, U, \dots, j_{i}}.
  \end{equation}
 This follows immediately from the fact that $\Ad(\dot{s}_{\epsilon_{i}})\fg_{\epsilon_{j}}=\fg_{\epsilon_{j}}$ for $j\neq i$, and $\Ad(\dot{s}_{\epsilon_{i}})\fg_{\pm\epsilon_{i}}=\fg_{\mp\epsilon_{i}}.$

 We now analyze the irreducible variety $\Xi_{j_{1},\dots, j_{i}}$.  Suppose that for the subsequence $1\leq k_{1}<\dots<k_{m-1}\leq i$ we have $j_{k_{1}}=j_{k_{2}}=\dots=j_{k_{m-1}}=L$ and that for the complementary subsequence $k_{m}<\dots< k_{i}$ we have $j_{k_{m}}=j_{k_{m+1}}=\dots=j_{k_{i}}=U$.  First, suppose that $ i <l$.  Consider the element 
\begin{equation}\label{eq:sigmaweyl}
\sigma:=s_{\epsilon_{k_{1}}}s_{\epsilon_{k_{2}}}\dots s_{\epsilon_{k_{m-1}}}\in W.
\end{equation} 
 It follows from Equations (\ref{eq:isinfr}) and (\ref{eq:LtoU}) that
  \begin{equation}\label{eq:infr}
 \Ad(\dot{\sigma}) \Xi_{j_{1},\dots, j_{i}}\subset \fr. 
 \end{equation}
 Note that $s_{\eps_{j}}$ acts on the coordinates of $\fh$ by sign change in the $j$-th coordinate.  Thus, if $m-1$ is even, it follows from part (1) of Proposition \ref{prop_weylk} that 
 $\sigma\in W_{K}$, and we can choose its representative
$\dot{\sigma}\in K$.  If $m-1$ is odd,
 then replace $\sigma$ by $\tau:=s_{\epsilon_{l}}s_{\epsilon_{k_{1}}}s_{\epsilon_{k_{2}}}\dots s_{\epsilon_{k_{m-1}}}$.
 Then we can choose $\dot{\tau}\in K$, and since we can choose $\dot{s}_{\epsilon_{l}}\in L_{ss}$, Equation (\ref{eq:infr}) implies
 $$
 \Ad(\dot{\tau}) \Xi_{j_{1},\dots, j_{i}}\subset \fr.
 $$
 In either case, the component $\Xi_{j_{1},\dots, j_{i}}$ is $W_{K}$-conjugate to 
 a subvariety of $\fr$, and Equation (\ref{eq:ilessl}) follows from (\ref{eq:upsunion}).  Now consider the case where $i=l$.  
 Choose $\sigma$ as in (\ref{eq:sigmaweyl}).  Then it follows from (\ref{eq:isinfr}) that $\Ad(\dot{\sigma})\Xi_{j_{1},\dots, j_{l}}\subset\fb_{+}$.  Now if $m-1$ is even, then $\Xi_{j_{1},\dots, j_{l}}\subset Y_{Q_{+}}=\Ad(K)\fb_{+}$.  However, if $m-1$ is odd, then $\Ad(\dot{\tau})\Xi_{j_{1},\dots, j_{l}}\subset s_{\epsilon_{l}}(\fb_{+})$, whence $\Xi_{j_{1},\dots, j_{l}}\subset Y_{Q_{-}}=\Ad(K)s_{\epsilon_{l}}(\fb_{+})$.  Thus, Equation (\ref{eq:iisl}) is proven.

We now prove (\ref{eq:Xicontainment}) when $\fg=\fso(2l,\C)$.  
By Theorem \ref{thm:YQs1}, it suffices to prove 
\begin{equation}\label{eq:XicontainmentD}
\Xi\subset Y_{Q_{\fr}},
\end{equation}
where $\fr$ is the parabolic subalgebra generated by $\fb_{+}$ and the negative simple root spaces
$\fg_{-\alpha_{i+1}},\dots, \fg_{-\alpha_{l}}$ for $i<l-1$, and $\fr=\fb_{+}$ for $i=l-1$.  
Recall from Example \ref{ex:roottypes} that
$$
\fg^{-\theta}=\displaystyle\bigoplus_{j=1}^{l-1}( \fg_{\epsilon_{j}-\epsilon_{l}}\oplus \fg_{\epsilon_{j}+\epsilon_{l}})^{-\theta}\oplus(\fg_{-(\epsilon_{j}-\epsilon_{l})}\oplus \fg_{-(\epsilon_{j}+\epsilon_{l})})^{-\theta} .$$
Let $e_{\pm j}$ be a basis for $( \fg_{\pm (\epsilon_{j}-\epsilon_{l})}\oplus \fg_{\pm(\epsilon_{j}+\epsilon_{l})})^{-\theta}$ respectively.  
Consider elements of the form 
\begin{equation}\label{eq:XirootspaceD}
\underline{a}\oplus_{j=1}^{l-1} u_{j}e_{j}\oplus_{j=1}^{l-1} v_{j} e_{-j},
\end{equation}
where $\underline{a}=\mbox{diag}[a_{1},\dots, a_{l}, -a_{l}, \dots, -a_{1}],$
$a_{i}\neq \pm a_{j}$, $a_{i}\neq 0$ for $i, \, j\leq l-1$,
and $u_{j},\, v_{j}\in \C$.  Arguing as in the previous case, we see $\Xi$ consists 
of elements of the form (\ref{eq:XirootspaceD}) satisfying
\begin{equation}\label{eq:conditionsD}
u_{j}v_{j}=0 \mbox{ for } j=1,\dots, i.  
\end{equation}
We define the varieties $\Xi_{j_{1},\dots, j_{i}}$ with
$j_{k}=L,\, U$ analogously to (\ref{eq:UorL}).  We have 
$\Xi=\bigcup_{j_{k}=L, U} \Xi_{j_{1},\dots, j_{i}}$ (cf. (\ref{eq:upsunion})).  
Now we observe that if $j_{k}=U$ for all $k$, 
then 
\begin{equation}\label{eq:isinfrD}
\Xi_{U,\dots, U}\subset\fr.
\end{equation}
This follows from the observation that 
$\epsilon_{j}-\epsilon_{l}=\alpha_{j}+\dots+\alpha_{l-1}$ 
and $\epsilon_{j}+\epsilon_{l}=\alpha_{j}+\dots +\alpha_{l-2}+\alpha_{l}$.  
Thus, for $j=i+1,\dots, l$, the root spaces $\fg_{\pm(\epsilon_{j}-\epsilon_{l})}$ and $\fg_{\pm(\epsilon_{j}+\epsilon_{l})}$ are in $\fl_{ss}\subset\fr$. Further,  
 for $j=1,\dots i$,  the root spaces 
$\fg_{\epsilon_{j}-\epsilon_{l}}$ and $\fg_{\epsilon_{j}+\epsilon_{l}}\subset\fu\subset\fr$ (see (\ref{eq:Levi1})).
We now show that any $\Xi_{j_{1},\dots, j_{i}}\subset Y_{Q_{\fr}}$.  
Recall from part (2) of Proposition \ref{prop_weylk} that
$$W^{\theta}=W_{K}=\langle s_{\alpha_{1}},\dots, s_{\alpha_{l-2}}, s_{\alpha_{l-1}}\cdot s_{\alpha_{l}}\rangle.$$
For $j=1,\dots, i$, define 
$w_{j}:= s_{\epsilon_{j}-\epsilon_{l-1}} s_{\alpha_{l-1}} s_{\alpha_{l}} s_{\epsilon_{j}-\epsilon_{l-1}}$.  
Then $w_{j}\in W_{K}$, and $w_{j}$ has order $2$.  In fact, $w_{j}$ acts on $\fh$ via
\begin{equation}\label{eq:actionwj}
w_{j}: (a_{1},\dots, a_{j}, \dots, a_{l})\to (a_{1},\dots, -a_{j},\dots, -a_{l}).  
\end{equation}
We claim that
\begin{equation}\label{eq:LtoUD}
\Ad(\dot{w}_{j})\Xi_{j_{1},\dots, j_{k-1}, L, \dots, j_{i}}\subset \Xi_{j_{1},\dots, U,\dots, j_{i}}. 
\end{equation}
Indeed, (\ref{eq:actionwj}) implies that 
$$w_{j}\cdot(\epsilon_{j}+\epsilon_{l})=-(\epsilon_{j}+\epsilon_{l}),\, w_{j}\cdot(\epsilon_{j}-\epsilon_{l})=-(\epsilon_{j}-\epsilon_{l}), \mbox{ and }
w_{j}\cdot(\epsilon_{k}+\epsilon_{l})=\epsilon_{k}-\epsilon_{l} \mbox{ for } k\neq j.$$
Further, since $w_{j}\in W_{K}$, 
$$
\Ad(\dot{w}_{j}): (\fg_{\pm (\epsilon_{j}-\epsilon_{l})}\oplus\fg_{\pm (\epsilon_{j}+\epsilon_{l})} )^{-\theta}\mapsto(\fg_{\mp(\epsilon_{j}-\epsilon_{l})}\oplus\fg_{\mp(\epsilon_{j}+\epsilon_{l})})^{-\theta}
$$
and  $\Ad(\dot{w}_{j})$ stabilizes the space $(\fg_{\pm(\epsilon_{k}-\epsilon_{l})}\oplus\fg_{\pm(\epsilon_{k}+\epsilon_{l})})^{-\theta}\,$
for $k \neq j$.
Equation (\ref{eq:LtoUD}) now follows from the definition of
the varieties $\Xi_{j_{1},\dots, j_{i}}$.  Thus, if we are given a variety $\Xi_{j_{1},\dots, j_{i}}$ with $j_{k_{1}}=j_{k_{2}}=\dots=j_{k_{m-1}}=L$, 
it follows from (\ref{eq:LtoUD}) and (\ref{eq:isinfrD}) that 
$$
\Ad(\dot{w}_{k_{1}}\dots \dot{w}_{k_{m-1}})\Xi_{j_{1},\dots, j_{i}}\subset\fr.
$$
Since $\Xi=\bigcup_{j_{k}=L, U} \Xi_{j_{1},\dots, j_{i}}$, Equation (\ref{eq:XicontainmentD}) follows.  
This completes the proof.
\end{proof}

\begin{proof}[Proof of Theorem \ref{thm:bigthm}]
Equation (\ref{eq:king}) follows from Theorems \ref{thm:YQs2} and \ref{thm:YQs1} along with Theorem 
\ref{thm:exhaustion}.  The statement about the number of irreducible components of $\fg(\geq i)$ follows 
from Parts 2 and 3 of Propositions \ref{prop_sooddflag} and \ref{prop_soevenflag}.  
\end{proof}


\begin{cor}\label{c:exactcoin}
Recall the variety $\fg(i)$ defined in Equation (\ref{eq:fgl}).  
The irreducible component decomposition of $\fg(i)$ is 
\begin{equation}\label{eq:fgidecomp}
\fg(i)=\displaystyle\bigcup_{\mbox{codim}(Q)=i} Y_{Q}\cap\fg(i).
\end{equation}
\end{cor}
\begin{proof}
Theorem \ref{thm:bigthm} and Equation (\ref{eq:king}) imply that the irreducible component decomposition
of the variety $\fg(i)$ is 
\begin{equation}\label{eq:firstdecomp} 
\fg(i)=\displaystyle\bigcup_{\mbox{codim}(Q)=i} \overline{Y_{Q}}\cap\fg(i),
\end{equation}
By Propositions \ref{prop_dimgl} (1) and \ref{prop:dimYQ}, we have 
$\overline{Y_{Q}}\cap\fg(i)\neq\emptyset$ for all $Q$ 
with $\mbox{codim}(Q)=i$.  For each $K$-orbit $Q$ 
with $\mbox{codim}(Q)=i$, we claim that
\begin{equation}\label{eq:closures}
\overline{Y_{Q}}\cap\fg(i)=Y_{Q}\cap \fg(i).
\end{equation}
  Indeed, suppose
that (\ref{eq:closures}) were false.  Then since 
$\overline{Y_{Q}}=\bigcup_{Q^{\prime}\subset\overline{Q}} Y_{Q^{\prime}}$ 
by Lemma \ref{l:YQclosure}, there exists a $K$-orbit $Q^{\prime}$ 
with $\mbox{codim}(Q^{\prime})> \mbox{codim}(Q)$ such that $Y_{Q^{\prime}}\cap\fg(i)\neq\emptyset.$ 
But this contradicts Theorem \ref{thm:bigthm} which asserts that $Y_{Q^{\prime}}\subset\fg(\geq i+1)$.  
Equation (\ref{eq:fgidecomp}) now follows from (\ref{eq:closures}) and (\ref{eq:firstdecomp}).
\end{proof}

The following corollary will be useful in Sections \ref{s:closedKorbits} and \ref{s:nilfibre}.

\begin{cor} \label{c:parabolics}
For $i=0, \dots, r_{n-1}-1$, the irreducible component decomposition of $\fg(i)$ is
\begin{equation}\label{eq:withparabolics}
\fg(i)=Y_{Q_{\fr}}\cap\fg(i),
\end{equation}
where $\fr$ is the $\theta$-stable parabolic subalgebra of Theorems \ref{thm:YQs2} and \ref{thm:YQs1}.
For $i=r_{n-1}$ and $\fg=\fso(2n,\C)$,
\begin{equation}\label{eq:oneclosed}
\fg(r_{n-1})=\fg(\geq r_{n-1})= Y_{Q_{+}},
\end{equation}
where $Q_{+}=K\cdot \fb_{+}$ is the unique closed $K$-orbit on $\B$ (see part (2) of Proposition \ref{prop_soevenflag}).
For $\fg=\fso(2n+1,\C)$ the irreducible component decomposition of $\fg(r_{n-1})$ is
\begin{equation}\label{eq:twoclosed}
\fg(r_{n-1})=\fg(\geq r_{n-1})= Y_{Q_{+}}\cup Y_{Q_{-}},
\end{equation}
where $Q_{+}$ and $Q_{-}$ are the distinct closed $K$-orbits on $\B$ (see part (2) of Proposition \ref{prop_sooddflag}).
\end{cor}
\begin{proof}
The result follows immediately from Equation (\ref{eq:firstdecomp}) and part (2) of Theorems \ref{thm:YQs2} and \ref{thm:YQs1}. 

\end{proof}

\section{The geometric invariant theory of multiplicity free spherical pairs}\label{ss:git}

In this section, we study the $K$-action on $\fg$ in the cases
$(K,\fg)=(GL(n-1,\C),\fgl(n,\C))$ and $(SO(n-1,\C),\fso(n,\C)).$  We extend a result of Kostant characterizing
regular elements using differentials in Theorem \ref{thm:Kostant}.
We then analyze the $K$-action on the subvariety $\fg(0)$, and show that
all the $K$-orbits in $\fg(0)$ are closed.    We use the above
analysis to give representatives of the closed $K$-orbits in $\fg$,
and discuss some applications to strongly regular elements.

\begin{dfn}\label{dfn:spherical}
Let $G$ be a reductive, algebraic group, and let $H\subset G$ be a
reductive algebraic subgroup.  The pair $(G, H)$ is called spherical if 
 $H$ acts on the flag variety $\B$ of $\fg$ 
with finitely many orbits. 
\end{dfn}
\begin{rem}\label{r:spherical2}
Let $V$ be a rational $G$-representation, and let $V^{H}$ be the set of $H$-fixed vectors in $V$.
It is well-known that Definition \ref{dfn:spherical} is equivalent to the statement 
that $\dim V^{H}\leq 1$ for every irreducible, rational $G$-representation $V$ (see \cite{KimelVin}, \cite{Brionclass}). 

\end{rem}

Let $(G, H)$ be a spherical pair.  Let $\fg=\mbox{Lie}(G)$ and 
let $\fh=\mbox{Lie}(H)$.  Let $\langle\langle\cdot, \cdot\rangle\rangle$ denote the Killing form on $\fg$, and let
$\fh^{\perp}$ be the annihilator of $\fh$ with respect to $\langle\langle\cdot, \cdot\rangle\rangle$.  Then the adjoint action of 
$G$ on $\fg$ restricts to an action of $H$ on $\fh^{\perp}$, which is referred to in the literature 
as the coisotropy representation of $H$ (see \cite{Pancoiso}).  Let $\C[\fh^{\perp}]^{H}$ be the 
ring of $H$-invariant polynomials on $\fh^{\perp}$.   Then it is well-known that $\C[\fh^{\perp}]^{H}$ is a polynomial algebra (Kor 7.2 of \cite{Kn} or Corollary
5 of \cite{Pancoiso}).  Consider the geometric invariant theory quotient $\Psi: \fh^{\perp}\to \fh^{\perp}//H$.  In Korollar 7.2 of \cite{Kn}, Knop proved
that $\Psi$ is flat.  
We consider spherical pairs satisfying: 
\begin{equation}\label{eq:numerology}
\dim\B=\dim \fh^{\perp}-\dim \fh^{\perp}// H.  
\end{equation}
In the appendix, we give a different and simpler proof of Knop's result
for spherical pairs satisfying (\ref{eq:numerology}) by using conormal
geometry.

We now analyze further what the condition in Equation (\ref{eq:numerology}) means
for the coisotropy representation.  If an algebraic group $A$ acts on
an irreducible variety $Y$, we say $y\in Y$ is $A$-regular if $\dim(A\cdot y)
\ge \dim(A\cdot z)$ for all $z\in Y$.  When the group $A$ is clear, we
let $Y_{reg}$ denote its $A$-regular elements.
Recall that an element 
$x\in\fg$ is $\Ad(G)$-regular if 
$\dim(\Ad(G)\cdot x)=\dim(\fg) - \mbox{rank}(\fg).$
A basic result of Kostant (Theorem 9,\cite{Kostant63}) states that if $\C[\fg]^{G}=\C[\psi_{1},\dots, \psi_{r}]$
is the ring of $\Ad(G)$-invariant polynomials on $\fg$, then 
\begin{equation}\label{eq:regdiffs}
x\in\fg_{reg} \mbox{ if and only if } d\psi_{1}(x)\wedge\dots\wedge d\psi_{r}(x)\neq 0. 
\end{equation}
 If $x\in\fg_{reg}$, and we identify $T_{x}^{*}(\fg)$ with $\fg$ using 
 the non-degenerate form on $\fg$, then 
 \begin{equation}\label{eq:centralizer}
 \mbox{span}\{ d\psi_{i}(x):\, i=1\,\dots, r\}=\fz_{\fg}(x), 
 \end{equation}
 where $\fz_{\fg}(x)$ denotes the centralizer of $x$ in $\fg$.  
We study the set of $H$-regular elements:
\begin{equation}\label{eq:Hreg}
\fh_{reg}^{\perp}=\{x\in \fh^{\perp}: \dim H\cdot x \mbox{ is maximal}\}.
\end{equation}
The following result relates the sets $\fh^{\perp}_{reg}$ and $\fg_{reg}$.  

\begin{thm}\label{thm:regelts}
Let $(G, H)$ be a spherical pair.  Then the following conditions are equivalent.
\begin{enumerate}
\item Equation (\ref{eq:numerology}) holds.
\item We have $\fh^{\perp}_{reg}\subset \fg_{reg}$. 
\end{enumerate}
\end{thm}

\begin{proof}
We first show that (1) implies (2).
Let $x\in \fh^{\perp}_{reg}.$   By Theorems 3 and 6 and Equation (15) 
of \cite{Pancoiso},
\begin{equation}\label{eq:zero}  
\dim \fh^{\perp}//H=\mbox{codim}_{\fh^{\perp}} H\cdot x.
\end{equation}
  By (1), 
\begin{equation}\label{eq:first}
\dim H\cdot x=\dim \B.
\end{equation}
 By Proposition 1 of \cite{Pancoiso}, we know that 
$\dim(\Ad(G) x) \ge 2 \dim(H\cdot x)
= 2\dim(\B).$   It follows that $x\in \fg_{reg}.$
For the converse,  by Theorem 3 of \cite{Pancoiso}, there is a dense
open subset $U$ of $\fh^{\perp}_{reg}$ such that if
$y\in U$, then $\dim(\Ad(G)y)= 2\dim(H \cdot y).$
Let $x\in U \subset \fg_{reg}$.  
Then $\dim(H\cdot x)=\frac{1}{2}\dim(\Ad(G) x)
=\dim(\B).$  The assertion now follows by Theorems 3 and 6 of \cite{Pancoiso}.
\end{proof}

Let $\theta$ be an involution of $\fg$.   It is well-known that
the pair $(\fg, \fk:= \fg^{\theta})$ is spherical \cite{Mat79, Sp}.
Recall that an involution $\theta$ of $\fg$ is called quasi-split if there
is a Borel subalgebra $\fb \in \B$ such that $\fb \cap \theta(\fb)$
is a Cartan subalgebra of $\fg$.  Let $K$ be the connected subgroup of
$G$ with Lie algebra $\fk$.  Let $\fp:=\fg^{-\theta}\cong \fk^{\perp}$.
Let $\fh$ be a $\theta$-stable Cartan subalgebra of $\fg$ such that
$\dim(\fh^{-\theta})$ is maximal among all $\theta$-stable Cartan
subalgebras of $\fg$.  We let $\fa=\fh^{-\theta}$, and let $\fm=\fz_{\fk}(\fa)$.
We let $\Phi_{c}$ be the compact roots for $\fh$ in $\fg$.

\begin{prop}\label{p:symmnum}
Let $\theta$ be an involution of $\fg$.  Then the  spherical
pair $(\fg,\fk)$ satisfies Equation (\ref{eq:numerology}) if and only
if $\theta$ is quasi-split.
\end{prop}

\begin{proof}  Let $x \in \fp$ be
$K$-regular.  By Equation
(\ref{eq:zero}), we know $\dim(\fp)-\dim(\fp//K)=\dim(K) - \dim(K_x)$.
    If $K\cdot \fb$ denotes the open $K$-orbit on $\B$ and $K_{\fb}$
is the stabilizer of $\fb$, then
 $\dim(\B)=\dim(K)-\dim(K_{\fb})$, so Equation (\ref{eq:numerology}) holds
if and only if $\dim(K_{\fb})=\dim(K_x)$. 
   By Proposition 8 of \cite{KR}, $\dim(K_x)
=\dim(\fm)$.  By Proposition 6.70 and page 394 of \cite{Knapp02},
it follows that $\dim(\fm)=\dim(\fh^{\theta}) + |\Phi_{c}|$.
By Corollary 2 of \cite{BH},
we know $\dim(K_{\fb})=\dim(\fh^{\theta}) + \frac{1}{2}|\Phi_{c}|$.
It follows that Equation (\ref{eq:numerology}) is equivalent
to the assertion that $\Phi_{c}$ is empty, which happens
if and only if $\theta$ is quasi-split by Lemma 8.3 of
\cite{RS}.
\end{proof}


\subsection{Kostant's Theorem for the $K$-action on $\fg$}
We now apply Theorem \ref{thm:regelts} to study the $K$-action on 
$\fg$ in the cases where $(K, \fg)=(GL(n-1,\C), \fgl(n,\C))$ or
$(SO(n-1,\C), \fso(n,\C)).$
  For $\fg=\fgl(n,\C)$, we view $\fk=\fgl(n-1,\C)$ as the top left hand corner of 
$\fg$.  Then $K$ is the corresponding algebraic subgroup of $G=GL(n,\C)$.  

\begin{nota}\label{nota:glcase}
The following extends notation from the $\fso(n,\C)$ case to the
$\fgl(n,\C)$ case.  We continue
to use the notation $r_{i}=\mbox{rank}(\fg_{i})$, so that
$r_{i}=i$ for $\fg_{i}=\fgl(i,\C)$.   The partial Kostant-Wallach map
$\Phi_{n}:\fgl(n,\C) \to \C^{n-1} \times \C^{n}$ is $\Phi_{n}(x)=(f_{i,j}(x))_{i=n-1, n; j=1, \dots, i}$, where $\C[\fg_{i}]^{G_{i}}=\C[f_{i,1}, \dots, f_{i,i}]$.
Proposition \ref{prop:flat} also holds in this case as was noted in Remark \ref{r:glflat}, and $\Phi_{n}$ is identified with the quotient morphism
$\fg \to \fg//K$.  
The varieties $\fg(\geq i), \,\fg(i), \, V^{r_{n-1}, r_{n}}(\geq i), \mbox{ and } V^{r_{n-1}, r_{n}}(i)$ are all defined
analogously.  See Section 3 of \cite{CEeigen} for details. 
 \end{nota}

To apply Theorem \ref{thm:regelts} and the theory of spherical varieties to this situation, we consider the following setup.  
Let $G$ be a connected, reductive algebraic group, let $K$ be a reductive, connected algebraic subgroup, and let $\fg$ 
and $\fk$ be their Lie algebras.  We say that the branching from $G$ to $K$ is multiplicity free if 
for every irreducible, finite dimensional rational $G$-representation $V$, and every 
irreducible, finite dimensional rational $K$-representation $W$ we have 
$\dim \mbox{Hom}_{K}(W, V)\leq 1$.  
Now let $\tilde{G}=G\times K$ and $K_{\Delta}\subset \tilde{G}$ be the diagonal copy
of $K$ in $G\times K$, i.e.
$$
K_{\Delta}:=\{(g, g): g\in K\}.  
$$
Consider the pair 
$(\tilde{G}, K_{\Delta})$ with Lie algebra pair $(\tilde{\fg}, \fk_{\Delta})$.
The following result is well-known.
\begin{prop}\label{prop:isspherical}
\begin{enumerate}
\item The pair $(\tilde{G}, K_{\Delta})$ is spherical 
if and only if the branching rule from $G$ to $K$ is multiplicity free. 
\item For the pairs $(G,K)=(GL(n,\C), GL(n-1,\C))$ and $(SO(n,\C), SO(n-1,\C))$,
$(\tilde{G}, K_{\Delta})$ is spherical.
\end{enumerate}
\end{prop}
\begin{proof}
The first statement follows by Theorem B of \cite{brundanspher}, together
with the easy observaton that a Borel subgroup $B_K$ of $K$ has
an open orbit on the flag variety $G/B$ of $G$ if and only if
$K_{\Delta}$ has an open orbit on $G/B \times K/B_K$.   The
second statement follows from the first statement and well-known
branching laws (see \cite{Johnson}).
\end{proof}
 A spherical pair $(\tilde{G}, K_{\Delta})$ satisfying the above property
is called a \emph{multiplicity free spherical pair}.  A result of Knop
shows that up to isogeny, these are essentially the only two multiplicity
free spherical pairs \cite{Knopenv}.
In the sequel, unless otherwise specified, we assume that
that $(\fg, \fk)=(\fgl(n,\C), \fgl(n-1,\C))$ or $(\fso(n,\C), \fso(n-1,\C)).$ 

It is easy to see that the restriction of $\langle\langle \cdot, \cdot\rangle\rangle$ to 
$\fk$ is non-degenerate.  Equip $\tilde{\fg}=\mbox{Lie}(\tilde{G})=\fg\oplus\fk$ 
with the non-degenerate invariant form $\langle \cdot, \cdot\rangle=\langle\langle \cdot, \cdot\rangle\rangle+(\langle\langle \cdot, \cdot\rangle\rangle)|_{\fk}$.  For $x\in \fg$, let $x=x_{\fk}+x_{\fp}$ with $x_{\fk} \in \fk$ and
$x_{\fp} \in \fk^{\perp}$.
An easy calculation shows that
$$
\fk_{\Delta}^{\perp}=\{(x, -x_{\fk}): x\in\fg, \, x_{\fk}\in \fk\}.
$$
  Note that $\fk_{\Delta}^{\perp}\cong \fg$ via the map $(x,-x_{\fk})\mapsto x$.
This isomorphism intertwines the coisotropy representation of $K_{\Delta}$ on $\fk_{\Delta}^{\perp}$ 
with the action of $K$ on $\fg$ via conjugation.  We can now use the geometry of spherical varieties, in particular
Theorem \ref{thm:regelts}, to study the geometry of the $K$-conjugation on $\fg$ and the partial Kostant-Wallach map 
$\Phi_{n}$ (see (\ref{eq:partial})).  

\begin{lem} \label{lem:dimest} Consider the multiplicity-free spherical pairs 
$(\tilde{G}, K_{\Delta})$. 
\begin{enumerate} \item  Equation (\ref{eq:numerology}) holds.
\item  $\dim(K)=\dim(K\cdot x)$ if and only if $x\in (\fk_{\Delta}^{\perp})_{reg}$.
\item \begin{equation}\label{eq:Kreg}
(\fk_{\Delta}^{\perp})_{reg}\cong\{x\in\fg: \fz_{\fk}(x_{\fk})\cap\fz_{\fg}(x)=0\}.
\end{equation} 
\end{enumerate}
\end{lem}

\begin{proof}
Equation (\ref{eq:numerology}) is equivalent to the routine
identity 
\begin{equation}\label{eq:flagvarieties}
\dim(\B_{\fg}) + \dim(\B_{\fk})=\dim(\fg)-r_{n}-r_{n-1} =\dim(\fg)-\dim(\fg//K).
\end{equation}
To prove the second assertion, let 
$x\in (\fk_{\Delta})^{\perp}_{reg}$.  Since (1) holds, we can
apply Equation (\ref{eq:first}) to conclude that $\dim (K\cdot x)=\dim (\B_{\fg})+\dim (\B_{\fk}).$  
The assertion now follows from (\ref{eq:flagvarieties}) and the simple observation 
\begin{equation}\label{eq:quotdim}
\dim(\fg)-r_{n}-r_{n-1}=\dim K.
\end{equation}
The second assertion implies that $(x,-x_{\fk}) \in (\fk_{\Delta}^{\perp})_{reg}$
if and only if $\fz_{\fk}(x_{\fk})\cap\fz_{\fk}(x)=0$.   The third assertion
now follows since 
$\fz_{\fk}(x_{\fk})\cap\fz_{\fk}(x)=\fz_{\fk}(x_{\fk})\cap\fz_{\fg}(x)$.
\end{proof}

We now describe the regular elements 
of the coisotropy representation of the spherical pairs $(\tilde{G}, K_{\Delta})$, which establishes an analogue of
Kostant's theorem.   Denote the generators of $\C[\fg]^K$ 
by $\{f_{n-1,1}, \dots, f_{n-1, r_{n-1}}; f_{n, 1}, \dots,
f_{n, r_n}\}$.  Let 
$$
\omega_{\fg//K}:=df_{n-1,1} \wedge \dots \wedge df_{n-1, r_{n-1}} \wedge df_{n, 1} \wedge 
\dots \wedge df_{n, r_n} \in \Omega^{r_{n-1} + r_n}(\fg).
$$

\begin{thm}\label{thm:Kostant}
$$
x\in(\fk_{\Delta}^{\perp})_{reg}\mbox{ if and only if } 
\omega_{\fg//K}(x)\neq 0.
$$
\end{thm}
\begin{proof}
We first suppose that $\omega_{\fg//K}(x)\neq 0$.  By Equation (\ref{eq:regdiffs}), it 
follows that $x$ is regular in $\fg$ and $x_{\fk}$ is regular in $\fk$.  
Equation (\ref{eq:centralizer}) then implies that 
$\fz_{\fk}(x_{\fk})\cap\fz_{\fg}(x)=0$, so $x\in (\fk_{\Delta}^{\perp})_{reg}$ by Equation (\ref{eq:Kreg}). 

Conversely, suppose $x\in(\fk_{\Delta}^{\perp})_{reg}$.  Then by Theorem \ref{thm:regelts} and part (1) of Lemma \ref{lem:dimest}, $(x, -x_{\fk})\in\tilde{\fg}_{reg}$.  
Thus, both $x\in\fg$ and $x_{\fk}\in\fk$ are regular.  Hence by Equation (\ref{eq:regdiffs}), 
\begin{equation}\label{eq:twoindep}
df_{n-1, 1}(x_{\fk})\wedge\dots\wedge df_{n-1, r_{n-1}}(x_{\fk})\neq 0\mbox{ and } df_{n,1}(x)\wedge\dots\wedge df_{n, r_{n}}(x)\neq 0.
\end{equation}
  Since $x\in(\fk_{\Delta}^{\perp})_{reg}$, $\fz_{\fk}(x_{\fk})\cap\fz_{\fg}(x)=0$ by Equation (\ref{eq:Kreg}).  
It now follows from (\ref{eq:twoindep}) and (\ref{eq:centralizer}) that 
$\omega_{\fg//K}(x)\neq 0$. 
\end{proof}

Theorem \ref{thm:Kostant} has an immediate corollary which 
is of interest in linear algebra.
\begin{cor}\label{c:linearalgebra}
Let $x\in\fg$ and suppose that $\fz_{\fk}(x_{\fk})\cap\fz_{\fg}(x)=0$.  
Then $x\in\fg$ and $x_{\fk}\in\fk$ are both regular.
\end{cor}
\begin{proof}
This follows by Equation (\ref{eq:Kreg}) and Theorem \ref{thm:Kostant}.
\end{proof}

 Elements of $(\fk_{\Delta}^{\perp})_{reg}$ regarded as elements of $\fg$ play 
 a major role in our study of the $K$-action on $\fg$, so we give them a special name. 
 \begin{dfn}\label{dfn:nsreg}
 An element $x\in\fg$ such that $\fz_{\fk}(x_{\fk})\cap \fz_{\fg}(x)=0$ is said 
 to be $n$-strongly regular.  We denote the set of $n$-strongly regular elements by $\fg_{nsreg}$.  
 \end{dfn}
 \begin{rem}\label{r:nsrfacts}
 In \cite{CEeigen}, we defined the set of $n$-strongly regular elements 
 for $\fg$ to be the set of elements $x\in\fg$ for which $\omega_{\fg//K}(x)\neq 0$.  
 It follows from Theorem \ref{thm:Kostant} and Equation (\ref{eq:Kreg})
 that our new definition is consistent with the previous one and $\fg_{nsreg}
\cong (\fk_{\Delta}^{\perp})_{reg}.$
 \end{rem}

We end this section by explaining how Corollary \ref{c:linearalgebra} can be used to simplify a crucial definition of 
Kostant and Wallach in the construction of the Gelfand-Zeitlin integrable system. 
We consider the chain of subalgebras
$$
\fg_{1}\subset \fg_{2}\subset\dots\subset\fg_{n-1}\subset \fg,
$$
where $\fg_{i}=\fgl(i,\C)$ (resp. $\fso(i,\C)$) when 
$\fg=\fgl(n,\C)$ (resp. $\fso(n,\C)$).  Let $G_{i}\subset G$ be the corresponding, connected algebraic group.  Let $\C[\fg_{i}]^{G_{i}}=\C[\psi_{i,1}, \dots, \psi_{i,r_{i}}]$ be the ring of $\Ad(G_{i})$-invariant polynomials on $\fg_{i}$, and let $\pi_{i}:\fg\to \fg_{i}$ be the projection off of $\fg_{i}^{\perp}$.  For 
$j=1,\dots, r_{i}$, define $f_{i,j}:=\pi_{i}^{*}\psi_{i,j}$.  Then the Gelfand-Zeitlin collection of functions is 
\begin{equation}\label{eq:GZfuns}
J_{GZ}:=\{f_{i,j}:\,i=1,\dots, n;\, j=1,\dots, r_{i}\}.  
\end{equation}
For $x\in\fg$, consider the subset of $T^{*}_{x}(\fg)$, 
$$
dJ_{GZ}(x):=\{df_{i,j}(x): \,i=1,\dots, n;\, j=1,\dots, r_{i}\}.
$$
The set $\fg_{sreg}$ of strongly regular elements was defined by
Kostant and Wallach to be the open set 
\begin{equation}\label{eq:sregelts}
\fg_{sreg}:=\{x\in\fg:\, dJ_{GZ}(x) \mbox{ is a linearly independent set}\}.
\end{equation}
Then $\fg_{sreg}$ is nonempty (Theorem 2.3 of \cite{KW1} for $\fgl(n,\C)$
and  Theorem 3.2 of \cite{Col2} for $\fg=\fso(n,\C)$).  
 The existence of a strongly regular element in a regular $G$-adjoint orbit
 implies 
that the functions $J_{GZ}$ in (\ref{eq:GZfuns}) form an integrable system on 
the orbit.

To characterize strongly regular elements, we need a little more notation.  For an $n\times n$ matrix $x$, let $x_{i}:=\pi_{i}(x)$, and let $\fz_{\fg_{i}}(x_{i})$ denote the centralizer of $x_{i}$ in $\fg_{i}$ thought of 
as a subalgebra of $\fg$.
\begin{prop}\label{prop:fullsreg}
An element $x\in\fg$ is strongly regular if and only if 
$$
\fz_{\fg_{i}}(x_{i})\cap\fz_{\fg_{i+1}}(x_{i+1})=0 \mbox { for } i=1,\dots, n-1.
$$
\end{prop}
 \begin{proof}
 An element $x\in\fg$ is strongly regular if and only if the following two conditions hold:
\begin{equation}\label{eq:theconditions}
\begin{split}
&(1)\; x_{i}\in\fg_{i}, x_{i+1} \in\fg_{i+1} \mbox{ are regular for all } i=1,\dots, n-1. \\
&(2)\; \fz_{\fg_{i}}(x_{i})\cap\fz_{\fg_{i+1}}(x_{i+1})=0 \mbox { for } i=1,\dots, n-1.
\end{split}
\end{equation}
For the case $\fg=\fgl(n,\C)$ this is the content of Theorem 2.14 of \cite{KW1}, and for $\fg=\fso(n,\C)$ it 
is Proposition 2.11 of \cite{Col2}.  It follows from Corollary \ref{c:linearalgebra} that if $x_{i+1}\in\fg_{i+1}$ 
satisfies (2) in (\ref{eq:theconditions}), then it automatically satisfies (1).  
\end{proof}

\subsection{The $K$-orbit structure of $\fg(0)$} \label{ss:generic}

We now study the $K$-orbit structure of the Zariski open subset 
$$
\fg(0)=\{x\in\fg: \sigma(x_{\fk})\cap\sigma(x)= \emptyset\}. 
$$
We show that $\fg(0)\subset\fg_{nsreg}$, and that each $K$-orbit in
$\fg(0)$ is closed in $\fg$.  The fact that $\fg(0)\subset\fg_{nsreg}$ follows from the following result in linear algebra.  
\begin{lem}\label{lem:linearalgebra} 
Let $V$ be a finite dimensional complex vector space.  Suppose we are
given a direct sum decomposition of $V$
$$
V=V_{1}\oplus V_{2}.
$$
 Let $X\in \mbox{End}(V)$, and let 
$Y\neq 0\in \mbox{End}(V)$ such that $Y:V_{1}\to V_{1}$ and $Y|_{V_{2}}=0$.  
Suppose that $[Y,X]=0$.  Then $X$ has a nonzero eigenvector $u\in V_{1}$.  
\end{lem}
\begin{proof}
The assumptions imply that the image $\Im(Y)$ of $Y$ is nonzero, contained
in $V_1$, and
stable under the action of $X$.   The result follows.
\end{proof}

The following consequence plays a crucial role in our study of $\fg(0)$.  
\begin{prop}\label{p:submatrix}
Let $X,\, Y$, and $V=V_{1}\oplus V_{2}$ be as in the 
statement of Lemma \ref{lem:linearalgebra}.  
Define $X_{1}: V_{1}\to V_{1}$ by $X_{1}:=\pi_{V_{1}}\circ X|_{V_{1}}$, where $\pi_{V_{1}}:V\to V_{1}$ 
is the projection onto $V_{1}$ off $V_{2}$.   Then $\sigma(X_{1})\cap \sigma(X)\neq \emptyset.$  
\end{prop}
\begin{proof} 
Let $u\in V_{1}$ be an eigenvector of $X$ of eigenvalue $\lambda$.  It follows from 
definitions that
$$
X_{1} u=\pi_{V_{1}} (Xu) =\pi_{V_{1}}(\lambda u)=\lambda u.
$$
Thus, $\lambda\in\sigma(X_{1})\cap \sigma(X).$
\end{proof}
\begin{exam}\label{ex:gln}
Let $V=\C^{n}$, and let $e_{1},\dots, e_{n}$ be the standard basis of 
$\C^{n}$.  Let $V_{1}=\mbox{span}\{e_{1},\dots, e_{k}\}$, and let 
$V_{2}=\mbox{span}\{e_{k+1},\dots, e_{n}\}$.  Let $x\in\fgl(n,\C)$, and let 
$x_{k}$ be the $k\times k$ submatrix in the upper lefthand corner of $x$.  
We embed $\fgl(k,\C)$ in $\fgl(n,\C)$ in the upper left corner.  Suppose there 
exists nonzero $Y\in \fgl(k,\C)$ with $[Y, x]=0$.  Then Proposition \ref{p:submatrix} implies that $\sigma(x_{k})\cap\sigma(x)\neq 0$. 
\end{exam}
We now return to the pairs $(\fgl(n,\C),\, \fgl(n-1,\C))$ and
$(\fso(n,\C),\, \fso(n-1,\C))$. 
Using Proposition \ref{p:submatrix}, we can prove a fundamental result regarding the structure of 
$\fg(0)$.  
\begin{thm}\label{thm:gzero}
Let $x\in\fg(0)$.  Then $x\in\fg_{nsreg}$.    
\end{thm}
\begin{proof}
Let $x\in\fg$ and suppose that $\fz_{\fk}(x_{\fk})\cap\fz_{\fg}(x)\neq 0$.  
We show that $x\in\fg(\geq 1)$ by considering the types $A, B, D$
separately. 
First, suppose that $\fg$ is type $A$.  Then decompose $V=\C^{n}$ as
$V=V_{1}\oplus V_{2}$ where $V_{1}=\mbox{span}\{e_{1},\dots, e_{n-1}\}$, 
and $V_{2}=\mbox{span}\{ e_{n}\}$.  Now apply Example \ref{ex:gln}.  
Similarly, when $\fg=\fso(2l,\C)$, we decompose 
$\C^{2l}$ as $V=V_{1}\oplus V_{2}$, where 
$V_{1}=\mbox{span} \{ e_{\pm 1}, \dots, e_{\pm (l-1)}, e_{l}+e_{-l}\}$, and 
$V_{2}=\mbox{span} \{e_{l}-e_{-l}\}$.  The reader can check that
 $\fk$ annihilates $V_{2}$.  Since the involution $\theta$ acts on 
$e_{\pm i}$ via $\theta(e_{\pm i})=e_{\pm i}$ for $i\neq l$ and $\theta(e_{l})=e_{-l}$ (see Section \ref{ss:symmetricreal}), we know $\theta$ acts on $V_{1}$ as the identity and on $V_{2}$
as the negative of the identity.  Therefore $x_{\fk}: V_{1}\to V_{1}$ and $x_{\fg^{-\theta}}: V_{1}\to V_{2}$, and it follows that $x_{1}=x_{\fk}$, where $x_1 = \pi_{V_1}(x|_{V_1})$. 
The result now follows from Proposition \ref{p:submatrix}.  The case
of $\fso(2l+1,\C)$ 
follows by taking $V_1 = \mbox{span} \{ e_{\pm 1}, \dots, e_{\pm l} \}$,
$V_2=\mbox{span} \{ e_0 \}$, and using Proposition \ref{p:submatrix}.
\end{proof}  


Let $c=(c_{r_{n-1}}, c_{r_{n}})\in V^{r_{n-1}, r_{n}}$ and write $c_{r_{i}}=(c_{i, 1},\dots, c_{i, r_{i}})\in \C^{r_{i}}$ for $i=n-1,\, n$.
 Let $I_{n, c}$ be the ideal of $\C[\fg]$ generated by the functions $f_{i, j}-c_{i,j}$ for $i=n-1, \, n$ and $j=1,\dots, r_{i}$.  
 \begin{cor}\label{c:isgenericallyrad}
 Let $c=(c_{r_{n-1}}, c_{r_{n}})\in V^{r_{n-1}, r_{n}}(0)$, so that $\Phi_n^{-1}(c)\subset\fg(0)$.  
 \begin{enumerate}
\item Then $I_{n,c}$ is radical, so that $I_{n,c}$ is the ideal of the fibre 
$\Phi_n^{-1}(c)$.
 Further, the variety $\Phi_{n}^{-1}(c)$ is smooth.
\item The fibre $\Phi_{n}^{-1}(c)$ is a single closed $K$-orbit.
\end{enumerate}\end{cor}

 \begin{proof}
 By Theorem \ref{thm:gzero} every element of the fibre $\Phi_{n}^{-1}(c)$ is $n$-strongly regular. It follows from Theorem \ref{thm:Kostant} and Remark
\ref{r:nsrfacts}
that the differentials 
 $\{df_{i,j}(x):\, i=n-1,\, n; \; j=1,\dots, r_{i}\}$ are independent for all $x\in \Phi_{n}^{-1}(c)$.   
By Theorem 18.15 (a) of \cite{Eis}, the ideal 
 $I_{n,c}$ is radical, so $I_{n,c}$ is the ideal of 
$\Phi^{-1}_{n}(c)$.  The smoothness
 of $\Phi^{-1}_{n}(c)$ now follows since the differentials of the generators of $I_{n,c}$ are independent at every point of $\Phi_{n}^{-1}(c)$.
For the second assertion, note first that 
$$\dim(K)=\dim(\fg)-\dim (\fg//K)=\dim( \Phi_{n}^{-1}(c)),$$
where the first equality follows from Equations (\ref{eq:flagvarieties}) and (\ref{eq:quotdim}), and the second equality follows from Proposition \ref{prop:flat} (3).  
By Lemma \ref{lem:dimest}, $\dim(K\cdot x)=\dim(K)$ for all 
$x\in \Phi_{n}^{-1}(c)$.  By Proposition \ref{prop:flat} (2),
each fibre $\Phi_{n}^{-1}(c)$ has a unique closed $K$-orbit, which
implies the assertion.
\end{proof} 

Using Theorem \ref{thm:gzero}, we can generalize a result of the first author to the orthogonal setting 
(cf the first statement of Theorem 5.15, \cite{Col1}).  
Consider the Zariski open subvariety of $\fso(n,\C)$
$$
\fso(n,\C)_{\Theta}:=\{x\in\fso(n,\C):\, \sigma(x_{i})\cap \sigma(x_{i+1})=\emptyset \mbox{ for } i=2,\dots, n-1\}.
$$
\begin{prop}\label{prop:orthofgtheta}
The elements of $\fso(n,\C)_{\Theta}$ are strongly regular.
\end{prop}
\begin{proof}
If $x\in \fso(n,\C)_{\Theta}$, then $x_{i}\in \fso(i,\C)(0)$ for $i=2,\dots, n$.  
It follows from Theorem \ref{thm:gzero} that $\fz_{\fg_{i-1}}(x_{i-1})\cap\fz_{\fg_{i}}(x_{i})=0$.  
The result now follows from Proposition \ref{prop:fullsreg}. 

\end{proof}
\begin{rem}\label{r:genericfibres}
The Gelfand-Zeitlin system for $\fgl(n,\C)$ is much better understood
than the Gelfand-Zeitlin system for $\fso(n,\C)$.
Let $J_{GZ}$ be the Gelfand-Zeitlin functions for either $\fg=\fgl(n,\C)$
or $\fg=\fso(n,\C)$
 defined in (\ref{eq:GZfuns}).  Consider
the Kostant-Wallach morphism:
\begin{equation}\label{eq:orthoKW}
\Phi:\fg \to \C^{r_1} \times \C^{r_{2}}\times\dots\times\C^{r_{n-1}}\times \C^{r_{n}} \mbox{ given by } \Phi(x)=( f_{i,j}(x))_{f_{i,j} \in J_{GZ}}  
\end{equation}
(for $\fg=\fso(n,\C)$, $\C^{r_1}$ is a point).
In \cite{KW1}, Kostant and Wallach prove that $\Phi$ is surjective,
and in Theorem 5.15 of \cite{Col1}, the first author shows that
for all $x$ in the Zariski open set
$$\fgl(n,\C)_{\Theta}:=\{x\in\fgl(n,\C):\, \sigma(x_{i})\cap \sigma(x_{i+1})=\emptyset \mbox{ for } i=2,\dots, n-1\},$$
 the fibre 
$\Phi^{-1}(\Phi(x))$ is irreducible.   
In later work, we will use the flatness assertion of Proposition \ref{prop:flat}
to show that $\Phi$ is surjective in the orthogonal case, which
together with the preceding proposition, shows that every regular
adjoint orbit contains strongly regular elements.  This implies that 
the Gelfand-Zeitlin functions in (\ref{eq:GZfuns}) form an integrable system on every regular adjoint orbit 
in $\fso(n,\C)$.  We will also use 
Proposition \ref{prop:orthofgtheta} and part (2) of 
Corollary \ref{c:isgenericallyrad} to show that $\Phi^{-1}(\Phi(x))$ is irreducible for $x\in\fso(n,\C)_{\Theta}$.  
The proofs of both these results for $\fso(n,\C)$ are different
and more conceptual than the analogous proofs for $\fgl(n,\C)$,
and we will develop these ideas in further work on the orthogonal
Gelfand-Zeitlin system.
\end{rem}





\subsection{Classification of closed $K$-orbits on $\fg$} \label{s:closedKorbits}

In Section \ref{ss:generic}, we showed that $K$-orbits in $\fg(0)$ are closed. In this section,
we describe the other closed $K$-orbits in $\fg$.  Our main tool is Theorem \ref{thm:bigthm} when $(\fg, K)=(\fso(n,\C), SO(n-1,\C))$
and Theorem 3.7 of \cite{CEeigen} when $(\fg, K)=(\fgl(n,\C), K=GL(n-1,\C))$.
  Recall the varieties 
$\fg(i)=\fg(\geq i)\setminus\fg(\geq i+1)$ defined in (\ref{eq:fgl}) and the partition $\fg=\bigcup_{i=0}^{r_{n-1}} \fg(i)$ of $\fg$ in (\ref{eq:fgipart}). For 
the analogous definition in type A, see Equation (3.3) of \cite{CEeigen}.)
\begin{thm}\label{thm:closedKinfg}
Let $x\in\fg(i)$, $i=0,\dots, r_{n-1}$.  Then $K\cdot x$ is closed if and only if $K\cdot x\cap\fl(i)\neq\emptyset$, where $\fl(i):=\fg(i)\cap\fl$ and $\fl$ is a $\theta$-stable Levi subalgebra of 
the following form:

%
\begin{enumerate}
\item If $\fg=\fgl(n,\C)$, then $\fl$ is the Levi subalgebra of 
block diagonal matrices $$\fl= \fgl(n-i,\C)\oplus\fgl(1,\C)^{i}.$$
\item If $\fg=\fso(2l+1, \C)$, then $\fl$ is the $\theta$-stable
Levi subalgebra defined in Theorem \ref{thm:YQs2}.
\item If $\fg=\fso(2l, \C)$, then $\fl$ is the $\theta$-stable Levi
subalgebra defined in Theorem \ref{thm:YQs1}.


\end{enumerate}
\end{thm}

Theorem \ref{thm:closedKinfg} will follow from two lemmas.
\begin{lem}\label{lem:inclosure}
Let $x\in\fg(i)$, and let $\fl$ be the corresponding
 Levi subalgebra in Theorem \ref{thm:closedKinfg}.  
Then $\overline{K\cdot x}$ contains an 
element of $\fl(i)$.
\end{lem}
\begin{proof}
Any element $x\in\fg(i)$ is $K$-conjugate to an element in a $\theta$-stable parabolic subalgebra 
$\fr$ with Levi factor $\fl$.  This follows from Corollary \ref{c:parabolics} when 
$\fg=\fso(n,\C)$ and from Theorem 3.7 of \cite{CEeigen} when $\fg=\fgl(n,\C)$.  
Thus, we can assume that $x\in\fr$.  
We choose an element $z$ in the centre $\fz$ of $\fl$ such that $\alpha(z)>0$ for every
root $\alpha$ of $\fu$, the nilradical of $\fr$.   Note that $\fz
\subset \fk$, which is clear for type A, and follows by Remark \ref{r:centerink}
for the orthogonal cases.
Then
\begin{equation}\label{eq:flclosure}
\displaystyle\lim_{t\to-\infty}\Ad(\exp tz) x\in\fl\cap\overline{K\cdot x}=\fl(i)\cap\overline{K\cdot x},
\end{equation}
where the last equality follows since $\overline{K\cdot x}\subset \Phi^{-1}(\Phi(x))\subset \fg(i)$.
\end{proof}
We now study the $K$-orbits of elements in $\fl$.  
\begin{lem}\label{lem:elementsinl}
Let $\fl$ be one of the Levi subalgebras in Theorem \ref{thm:closedKinfg}.  Any two 
elements in $\fl(i)$ which lie in the same fibre of the partial Kostant-Wallach map $\Phi_{n}$ are 
$K$-conjugate. 
\end{lem}
\begin{proof}
Suppose that $x,\, y\in \Phi_{n}^{-1}(c)$, with $c\in V^{r_{n-1}, r_{n}}(0)$.  Then 
Corollary \ref{c:isgenericallyrad} implies that $x$ and $y$ are $K$-conjugate.  

Now suppose that $x,\, y\in \Phi_{n}^{-1}(c)\cap\fl$ with $c\in V^{r_{n-1}, r_{n}}(i)$ with $i>0$.  
Decompose $x$ and $y$ as 
$x=x_{\fz}+x_{\fl_{ss}}$ and $y=y_{\fz}+y_{\fl_{ss}}$, with 
$x_{\fz}, y_{\fz}\in\fz$ and $x_{\fl_{ss}},\, y_{\fl_{ss}}\in\fl_{ss}$.  
Then $\sigma(x)\cap \sigma(x_{\fk})$ are the coordinates of $x_{\fz}$ 
and similarly for $y$.  Since $\Phi_{n}(x)=\Phi_{n}(y)\in V^{r_{n-1}, r_{n}}(i)$, we know
$ \sigma(x)\cap \sigma(x_{\fk})=\sigma(y)\cap \sigma(y_{\fk})$.  
It follows that there exists $\dot{w}\in N_{K} (L\cap K)$ such that 
$\Ad(\dot{w})x_{\fz}=y_{\fz}$.  So without loss of generality, we may assume that 
$x_{\fz}=y_{\fz}$.  Since $x, y \in \fg(i)\cap \fl$,  then 
$x_{\fl_{ss}},\, y_{\fl_{ss}}\in\fl_{ss}(0)$, where 
$$
\fl_{ss}(0):=\{z\in\fl_{ss}:\; \sigma(z)\cap \sigma(z_{\fk})=\emptyset\}.  
$$
Let $\Phi_{\fl_{ss}}:\fl_{ss}\to \C^{\mbox{rk}(\fl_{ss}\cap \fk)}\times \C^{\mbox{rk}(\fl_{ss})}$ be 
the partial Kostant-Wallach map for $\fl_{ss}$.  Then $\Phi_{n}(x)=\Phi_{n}(y)$ implies 
that $\Phi_{\fl_{ss}}(x_{\fl_{ss}})=\Phi_{\fl_{ss}}(y_{\fl_{ss}}).$  But then Corollary \ref{c:isgenericallyrad} applied 
to $\fl_{ss}$ forces $x_{\fl_{ss}}$ and $y_{\fl_{ss}}$ to lie in the same $K\cap L_{ss}$-orbit.  
This completes the proof.  

\end{proof}


We now prove Theorem \ref{thm:closedKinfg}.  

\begin{proof}[Proof of Theorem \ref{thm:closedKinfg}]
Suppose that $x\in\fg(i)$ with $\Ad(K)\cdot x$ closed.  Then by 
Lemma \ref{lem:inclosure}, there exists an element $z\in\fl(i)\cap \overline{\Ad(K)\cdot x}$.  
But since $\Ad(K)\cdot x$ is closed, we conclude that  $\Ad(K)\cdot x=\Ad(K)\cdot z$.  

Conversely, suppose that $x\in\fl(i)$ and consider $\overline{\Ad(K)\cdot x}$.  
By Lemma \ref{lem:inclosure}, there exists $z\in\fl(i)$ with $K\cdot z$ closed and 
$K\cdot z\subset\overline{\Ad(K)\cdot x}\subset \Phi^{-1}(\Phi(x))$.  Therefore $\Ad(K)\cdot x=\Ad(K)\cdot z$ by Lemma 
\ref{lem:elementsinl}.  Thus, $\Ad(K)\cdot x$ is closed. 

\end{proof}

\subsection{The nilfibre of the partial Kostant-Wallach map in the orthogonal case}  \label{s:nilfibre}
Though there are many similarities between the $GL(n-1,\C)$-action 
on $\fgl(n,\C)$ and the $SO(n-1,\C)$-action on $\fso(n,\C)$, 
in this subsection we show that their $n$-strongly regular sets
are different.
In the case of $\fgl(n,\C)$, every fibre of the partial Kostant-Wallach map contains $n$-strongly regular elements.  This follows 
easily from Theorem 2.3 of \cite{KW1}.  However, this is not the case for $\fso(n,\C)$.  To see this, we need 
to study the nilfibre $\Phi_{n}^{-1}(0,0)$ of the orthogonal partial Kostant-Wallach map in more detail using 
Theorems \ref{thm:bigthm} and \ref{thm:Kostant}.  
\begin{thm}\label{thm:nilfibre}
Let $\Phi_{n}:\fso(n,\C)\to \C^{r_{n-1}}\oplus \C^{r_{n}}$ be the orthogonal 
partial Kostant-Wallach map $\Phi_{n}$ defined in Equation (\ref{eq:partial}).  

Case I: Suppose $n=2l$.  Then $\Phi_{n}^{-1}(0,0)=K\cdot\fn_{+}$ is
irreducible, 
where $\fn_{+}=[\fb_{+},\fb_{+}]$ and $Q_{+}=K\cdot\fb_{+}$ is the unique closed $K$-orbit in $\B$ (see part (2) of Proposition \ref{prop_soevenflag}).

Case II: Suppose $n=2l+1$.  Then $\Phi_{n}^{-1}(0,0)=K\cdot\fn_{+} \cup K\cdot \fn_{-}$ has two irreducible components,
where $\fn_{\pm}=[\fb_{\pm},\fb_{\pm}]$ and $Q_{\pm}=K\cdot\fb_{\pm}$, are the two closed $K$-orbits in $\B$ (see part (2) of Proposition \ref{prop_sooddflag}).
\end{thm}
\begin{proof}
Let $Q=K\cdot \fb$ be a closed $K$-orbit.  Let $\fn=[\fb,\fb]$ be the nilradical 
of $\fb$.  We first show that $\Ad(K)\cdot\fn$ is an irreducible component of $\Phi_{n}^{-1}(0,0)$.  
Since $Q$ is closed, $\fb$ is $\theta$-stable by Proposition 4.12 of \cite{CEexp}.
Thus, $\fb\cap\fk$ is a Borel subalgebra of $\fk$ with nilradical $\fn\cap\fk$.  
It follows that for any $x\in\fn$, we have $\Phi_{n}(x)=(0,0).$  By the $K$-equivariance 
of $\Phi_{n}$, $\Ad(K)\cdot \fn\subset \Phi_{n}^{-1}(0,0)$.    

Recall the Grothendieck resolution $\widetilde{\fg}=\{(x,\fb): x\in\fb\}\subset\fg\times\B$ and 
the morphisms $\pi:\widetilde{\fg}\to \B$, $\pi(x,\fb)=\fb$ and $\mu:\widetilde{\fg}\to \fg$, $\mu(x,\fb)=x$. 
 Corollary 3.1.33 of \cite{CG} gives a $G$-equivariant isomorphism 
$\widetilde{\fg}\cong G\times_{B}\fb$.  Under this isomorphism 
$\pi^{-1}(Q)$ is identified 
with the closed subvariety $K\times_{K\cap B}\fb\subset G\times_{B}\fb$.  
The closed subvariety $K\times_{K\cap B}\fn\subset K\times_{K\cap B}\fb$ maps surjectively under $\mu$
to $\Ad(K) \fn$.  Since $\mu$ is proper, $\Ad(K) \fn$ is closed and irreducible.  
By Proposition 3.2.14 of \cite{CG}, the restriction of $\mu$ to 
$K\times_{K\cap B}\fn$ generically has finite fibres.  Thus, the same reasoning that we used in Equation (\ref{eq:YQdim}) along with Propositions \ref{prop:dimYQ} and \ref{prop_dimgl} shows that 
\begin{equation*}\label{eq:nulldim}
\begin{split}
\dim \Ad(K) \fn&=\dim(K\times_{K\cap B}\fn)\\
&=\dim(Y_{Q})-r_{n}\\
&=\dim(\fg(\geq r_{n-1}))-r_{n}\\
&=\dim(\fg)-r_{n-1}-r_{n}\\
&=\dim\Phi_{n}^{-1}(0,0). 
\end{split}
\end{equation*}
  Thus, 
$\Ad(K)\cdot \fn$ is an irreducible component of $\Phi_{n}^{-1}(0,0)$. 

We now show that every irreducible component of $\Phi_{n}^{-1}(0,0)$ is of the form $\Ad(K) \fn$. 
It follows from definitions
that $\Phi_{n}^{-1}(0,0)\subset \fg(r_{n-1})\cap \mathcal{N}$, where $\mathcal{N}\subset\fg$ is the nilpotent cone in $\fg$.  
Thus, if $\X$ is an irreducible component of $\Phi_{n}^{-1}(0,0)$, 
then $\X\subset \Ad(K) \fn$ by Equations (\ref{eq:oneclosed}) and (\ref{eq:twoclosed}) from Corollary \ref{c:parabolics}.  But then 
$\X=\Ad(K) \fn$ by Proposition \ref{prop:flat}. 
\end{proof}

We use Theorem \ref{thm:nilfibre} to study 
$\Phi_{n}^{-1}(0,0)$ in more detail.  In \cite{CEeigen}, we studied the interaction between the set of $n$-strongly regular 
  elements for the pair $(\fg=\fgl(n,\C),  \, \fk=\fgl(n-1,\C))$ and the nilfibre of the corresponding partial Kostant-Wallach map.  
We now show that unlike in the case of $\fgl(n,\C)$, there are no $n$-strongly regular elements in the nilfibre 
of the partial Kostant-Wallach map for the orthogonal Lie algebra $\fso(n,\C)$.  The key observation is the following 
proposition,  which can be viewed as an extension of Proposition 3.8 in \cite{CEKorbs}. 

\begin{prop}\label{prop:overlaps}  
Let $n>3$,
let $\fg=\fso(n,\C)$, and let $K=SO(n-1,\C)$.  Let $\fb\subset\fg$ be a Borel subalgebra and suppose that the
$K$-orbit $K\cdot\fb$ is closed in $\B$.  Let $\fn=[\fb,\fb]$ be the nilradical of $\fb$.  Then
\begin{equation}\label{eq:fncentralizers}
\fz_{\fk}(\fn\cap\fk)\cap\fz_{\fg}(\fn)\neq 0,
\end{equation}
where $\fz_{\fg}(\fn)$ is the centralizer of $\fn$ in $\fg$, and $\fz_{\fk}(\fn\cap\fk)$ is the centralizer
of $\fn\cap\fk$ in $\fk$. 
\end{prop}
\begin{proof}
Consider a closed $K$-orbit $Q$ in $\B$.  By $K$-equivariance, it suffices to show 
Equation (\ref{eq:fncentralizers}) for a representative $\fb$ of $Q$.  By part (2) of Propositions \ref{prop_sooddflag}
and \ref{prop_soevenflag}, we can assume that the standard diagonal Cartan subalgebra $\fh$ is in $\fb$.  
Let $\phi \in \Phi^{+}(\fg, \fh)$ be the highest root of $\fb$.  We claim for $n>4$ that $\phi$ is compact imaginary.  It then follows that the root space
$$
\fg_{\phi}\subset \fz_{\fk}(\fn\cap\fk)\cap\fz_{\fg}(\fn).
$$  

Suppose first that $\fg=\fso(2l,\C)$. By part (2) of Proposition \ref{prop_soevenflag}, we can 
assume that $\fb=\fb_{+}$.  The highest root is then $\epsilon_{1}+\epsilon_{2}$, which is compact imaginary 
for $l>2$ (Example \ref{ex:roottypes}.)  If $\fg=\fso(2l+1,\C)$, then by part (2) of Proposition \ref{prop_sooddflag}, we can assume 
that $\fb=\fb_{+}$ or $\fb=\fb_{-}=s_{\alpha_{l}}(\fb_{+})$.  In both cases, the highest root is $\epsilon_{1}+\epsilon_{2}$, 
which is compact imaginary (Example \ref{ex:roottypes}).  

If $\fg=\fso(4,\C)$, then $\phi=\epsilon_{1}+\epsilon_{2}$ is 
complex $\theta$-stable.  Since $\fn$ is abelian in this case,
 $(\fg_{\phi} \oplus \fg_{\theta(\phi)})^{\theta}\subset \fz_{\fk}(\fn\cap\fk)\cap \fz_{\fg}(\fn)$.

\end{proof}

\begin{cor}\label{c:nonsreg}
Let $n>3$, and let $\Phi_{n}:\fso(n,\C)\to \C^{r_{n-1}}\oplus \C^{r_{n}}$ be the orthogonal partial Kostant-Wallach map.    
Then $\Phi_{n}^{-1}(0,0)$ contains no $n$-strongly regular elements. 
\end{cor}
\begin{proof}
Suppose $x \in \Phi_{n}^{-1}(0,0)$, so by Theorem \ref{thm:nilfibre},
$x$ is contained in $\fn$, the nilradical of a Borel subalgebra
$\fb$ with $K\cdot \fb$ closed.
    By Proposition
\ref{prop:overlaps}, there is
a nonzero element $y$ of $\fz_{\fk}(\fn\cap \fk) \cap \fz_{\fg}(\fn)$.
Then $y\in\fz_{\fk}(x_{\fk})\cap \fz_{\fg}(x) $, so $x$ is
not $n$-strongly regular.
\end{proof}

\begin{rem}\label{r:lowdim}
The assertion of the corollary is false
 for $n=3$.  In this case,  $\fso(3,\C)\cong \fsl(2,\C)$ and $\fk=\fh\subset \fsl(2,\C)$, 
where $\fh$ is the standard Cartan subalgebra of $\fsl(2,\C)$.  
In this case, it follows by Proposition 3.11 from \cite{CEKorbs}
that each irreducible component contains strongly regular elements.
\end{rem} 


The following result is analogous to Proposition 3.11 of \cite{CEeigen}.  
We let $I_{n}$ be the ideal of $\C[\fso(n,\C)]$ generated
by elements of $\C[\fso(n,\C)]^{SO(n-1,\C)}$ of positive degree.  
\begin{cor}\label{c:notrad}
The ideal $I_{n}$ is radical if and only if $n=3$.  
\end{cor}
\begin{proof}
By Theorem 18.15 (a) of \cite{Eis}, the ideal $I_{n}$ is radical
if and only if the set of differentials $\{df_{i,j}(x): \, j=1,\dots, r_{i}, \, i=n-1, n\}$ is linearly independent on an open, dense subset 
of each irreducible component of $\Phi_{n}^{-1}(0,0)$.  It follows from Definition \ref{dfn:nsreg} and Theorem \ref{thm:Kostant} that $I_{n}$ is radical 
if and only if each irreducible component of $\Phi_{n}^{-1}(0,0)$ contains $n$-strongly regular elements.  
But it follows from Corollary \ref{c:nonsreg} and the case of $\fso(3,\C)$ in Remark \ref{r:lowdim}
 that each irreducible component of $\Phi_{n}^{-1}(0,0)$ contains $n$-strongly regular elements if and only if $n=3$.     
\end{proof}



Note that we have derived results concerning the $n$-strongly regular
set without using a slice, in contrast to the case of $\fgl(n,\C)$
studied by Kostant and Wallach \cite{KW1}, Theorem 2.3.

\begin{rem}\label{r:nosreg}
Consider the orthogonal Kostant-Wallach map $\Phi$ defined in (\ref{eq:orthoKW}).  It follows 
from Corollary \ref{c:nonsreg} that the nilfibre $\Phi^{-1}(0,\dots, 0)$ contains no strongly regular elements.
This is very different than the case of $\fgl(n,\C)$ studied extensively in \cite{CEKorbs}. 

\end{rem}



\section{appendix}
In the appendix, we prove a general result which implies Proposition 
\ref{prop:flat}.   The proof is an adaptation of the proof of Proposition
2.3 from \cite{CEeigen}.


\begin{thm}\label{thm:flatnessgen}
Let $(G,H)$ be a spherical pair such that 

$$
\dim\B=\dim \fh^{\perp}-\dim \fh^{\perp}// H.  
$$
(cf Equation (\ref{eq:numerology})).
Then $\Psi: \fh^{\perp}\to \fh^{\perp}//H$ is flat.  
\end{thm}
\begin{proof}
We first show that $\Psi^{-1}(0)$ is equidimensional of dimension
$\dim \fh^{\perp}-\dim \fh^{\perp}// H.$   Let $C$ be an irreducible
component of $\Psi^{-1}(0)$.  By standard results, $\dim(C) \ge
\dim \fh^{\perp}-\dim \fh^{\perp}// H.$  Label the finite number
of $H$-orbits on $\B$ by $Q_1, \dots, Q_s$.  Let 
$Z=\cup_{i=1}^s \overline{T_{Q_i}^{*}(\B)}$ and note that the irreducicible
components of $Z$ are the subvarieties $Z_i := \overline{T_{Q_i}^{*}(\B)}$,
and also that $\dim Z_i = \dim \B$ for $i=1, \dots, s$.   Recall
the standard identification $T^*\B = \{ (\fb, x) \in \B \times \fg :
x \in [\fb, \fb]\}$ and let $\mu:T^*\B \to \fg$ be the moment map,
$\mu(\fb, x)=x$.   

We claim that $\Psi^{-1}(0) \subset \mu(Z)$.   Indeed, by Theorem
6 of \cite{Pancoiso} , $\C[\fh^\perp]^H=\C[g_1, \dots, g_k]$
is a polynomial ring in $k$ generators, which can be taken 
to be homogeneous.   Further, the morphism $\Psi:\fh^{\perp} \to \fh^\perp//H$
may be identified with $(g_1, \dots, g_k):\fh^{\perp} \to \C^k$.
For $f\in \C[\fg]^G$ of positive degree, note that $f|_{\fh^\perp}
\in \C[\fh^\perp]^H$, so $f|_{\fh^\perp}$ is a polynomial of strictly
positive degree in the variables $g_1, \dots, g_k$.   By the
above identification, $g_i(x)=0$ for each $x\in \Psi^{-1}(0)$,
so $f(x)=0$.   By Proposition 16 of \cite{Kostant63}, it follows
that $x$ is nilpotent, and hence lies in $\fn:=[\fb, \fb]$, the
nilradical of a Borel subalgebra $\fb$.   Thus, if $Q_i=H\cdot \fb$,
then $(\fb, x)\in Z_i$, and $x=\mu(\fb,x)\in \mu(Z)$.

Since $\mu$ is proper, it follows that $C \subset \mu(Z_i)$ for some
$i$.  Hence, $\dim C \le \dim Z_i = \dim \B = \dim \fh^{\perp} - 
\dim \fh^{\perp}//H,$ so $\dim C = \dim \fh^{\perp} - \dim \fh^{\perp}//H.$

Now for $x\in \fh^{\perp}$,
let $d_x$ be the maximum of the dimension
of the irreducible components of $\Psi^{-1}( \Psi(x))$.   Since the
functions $g_1, \dots, g_k$ are homogeneous, scalar multiplication
by $\lambda \in \C^*$ induces an isomorphism $\Psi^{-1}(\Psi(x))
\cong \Psi^{-1}(\Psi(\lambda x)),$ so $d_x = d_{\lambda x}.$   By
upper semi-continuity of dimension, the set $\{ y\in \fh^{\perp} :
d_y \ge d \}$ is closed for each integer $d$ (Proposition 4.4
of \cite{Hum}).  Hence,  $d_y \le d_0 = \dim \fh^{\perp} - \dim \fh^{\perp}//H$
for all $y\in \fh^{\perp}$.  It follows that $d_y = \dim \fh^{\perp} - \dim \fh^{\perp}//H$ for all $y\in \fh^{\perp}$.   Hence, $\Psi$ is flat by the corollary
to Theorem 23.1 of \cite{Mat}.
\end{proof}








\bibliographystyle{amsalpha.bst}

\bibliography{bibliography-1}

\end{document}